\documentclass[leqno,12pt,oneside]{amsart}
\usepackage{amsmath,amssymb,amsthm,mathrsfs}
\usepackage[foot]{amsaddr}
\usepackage{epsfig,color}
\usepackage[latin1]{inputenc}
\usepackage[english]{babel}
\usepackage{mathtools}
\usepackage{braket}
\usepackage[toc]{appendix}
\usepackage{enumitem}
\usepackage{ulem}
\usepackage{graphicx}
\usepackage{xfrac, nicefrac}
\usepackage{csquotes}
\allowdisplaybreaks
\usepackage{esint}
\usepackage{hyperref}

\usepackage[margin=1in]{geometry}

\numberwithin{equation}{section}

\def\R{\mathbb R}
\def\N{\mathbb N}

\newcommand{\B}{\mathcal{B}}
\newcommand{\K}{\mathcal{K}}
\newcommand{\C}{\mathcal{C}}
\renewcommand{\H}{\mathcal{H}}
\def\om{\omega}
\def\Om{\Omega}

\newcommand{\rn}{\mathbb{R}^n}
\newcommand{\mres}{\mathbin{\vrule height 1.6ex depth 0pt width
0.13ex\vrule height 0.13ex depth 0pt width 1.3ex}}

\newtheorem*{theorem*}{Theorem}
\newtheorem{theorem}{Theorem}[section]
\newtheorem{lemma}[theorem]{Lemma}
\newtheorem{proposition}[theorem]{Proposition}
\newtheorem*{proposition*}{Proposition}

\newtheorem{corollary}[theorem]{Corollary}
\newtheorem{remark}[theorem]{Remark}
\newtheorem*{remark*}{Remark}

\newtheorem{definition}[theorem]{Definition}
\newtheorem*{definition*}{Definition}

\newcommand{\todo}[1]{\text{\colorbox{yellow}{#1}}}

\title[Smooth approximation of Lipschitz domains and weak curvatures]{Smooth approximation of Lipschitz domains, weak curvatures and isocapacitary estimates}

\begin{document}


\begin{abstract}
We provide a novel approach to approximate bounded Lipschitz domains via a sequence of smooth, bounded domains.
The flexibility of our method  allows either inner or outer approximations of Lipschitz domains which also possess weakly defined curvatures, namely, domains whose boundary can be locally described as the graph of a function belonging to the Sobolev space $W^{2,q}$ for some $q\geq 1$. 
The sequences of approximating sets is also characterized by uniform isocapacitary estimates with respect to the initial domain $\Omega$.
\end{abstract}

\author{Carlo Alberto Antonini \textsuperscript{1}}

\address{\textsuperscript{1} Dipartimento di Matematica  ``Federigo Enriques'', Universit\`a di Milano, Via Cesare Saldini 50, 20133 Milano, Italy}
\email{carlo.antonini@unimi.it}
\urladdr{}

\maketitle

\section{Introduction}\label{sec:intro}
In this paper we are concerned with inner and outer approximation of bounded Lipschitz domains $\Omega$ of the Euclidean space $\rn$, $n\geq 2$.
Specifically, we construct two sequences of $C^\infty$-smooth bounded domains $\{\om_m\},\{\Om_m\}$ such that $\om_m\Subset \Om\Subset \Om_m$ for all $m\in\N$, which also satisfy natural covergence properties like, for instance,  in the sense of the Lebesgue measure and  in the sense of Hausdorff to $\Om$. 

Geometric quantities like a Lipschitz characteristic $\mathcal L_\Omega=(L_\Omega,R_\Omega)$ and the diameter $d_\Omega$ of the domain $\Omega$ are comparable to the corresponding ones of its approximating sets $\om_m, \Om_m$. Here, the constant $R_\Omega$ stands for the radius of the ball domains on which the boundary can be described as a function of $(n-1)$-variables-- i.e. the local boundary chart-- and $L_\Omega$ is their Lipschitz constant-- see Section \ref{sec:duue} for    the precise definition of a Lipschitz characteristic of $\Omega$. 

Furthermore,  the smooth charts locally describing the boundaries $\partial \om_m,\partial \Om_m$ are defined on the same reference systems as the local charts describing $\partial \Om$, together with strong convergence in the Sobolev space $W^{1,p}$ for all $p\in [1,\infty)$.

If in addition the local charts describing $\partial \Om$ belong to the Sobolev space $W^{2,q}$ for some $q\in [1,\infty)$, then we also have strong convergence in the $W^{2,q}$-sense. In a certain way, this means that the second fundamental forms $\B_{\om_m}$ and $\B_{\Om_m}$ of the regularized sets converge in $L^q$ to the ``weak" 
 curvature $\B_\Om$ of the initial domain $\Om$.

\medskip

Smooth approximation of open  sets, not necessarily having Lipschitzian boundary, has been object of study by many authors. To the best of our knowledge, the first author who provided an approximation of this kind is  Ne\v{c}as \cite{necas}, followed by Massari \& Pepe \cite{massari} and Doktor \cite{doktor}. The underlying idea behind  their proof is nowadays standard, and it is typically used to approximate sets of finite perimeter. This consists in regularizing the characteristic function of $\Om$ via mollification and convolution, and then define  the approximating set $\Om_m$ as a suitable superlevel set of the mollified characteristic functions--see for instance \cite[Theorem 3.42]{afp} or \cite[Section 13.2]{maggi}. 
We point out that
Schmidt \cite{schmidt} and Gui, Hu \& Li \cite{gui} constructed smooth approximating domains \textit{strictly contained} in $\Omega$ under additional assumptions on the finite perimeter domain $\Om$ , whereas an outer approximation via smooth sets is given by Doktor \cite{doktor} when the domain $\Omega$ is endowed with a Lipschitz continuous boundary.

A different kind of approach, which makes use of Stein's regularized distance, has been recently developed by Ball \& Zarnescu \cite{ball}.  Here, the authors deal with $C^0$ domains, i.e. domains whose boundary can be locally described by merely continuous charts, and hence need not have finite perimeter. We mention that their regularized domains $\Om_\varepsilon$ are defined as the $\varepsilon$-superlevel set of the regularized distance function, which in turn is obtained via mollification of the usual signed distance function.
Here, the parameter $\varepsilon$ can be taken either positive or negative, according to the preferred method of approximation, whether from the inside or outside of $\Omega$.

\vspace{4pt}

The aforementioned techniques have thus been used to treat domains with ``rough" boundaries; however, they do not seem suitable to approximate domains which possess weakly defined curvatures, even in the case of domains having bounded curvatures, e.g. $\partial\Om\in C^{1,1}$.  
Namely, we do not recover any quantitative information or convergence property regarding the second fundamental forms $\B_{\Om_m}$ from the original one $\B_{\Om}$.
This is because first-order estimates regarding $\Om_m$ are proven by a careful pointwise analysis of the gradient of the local charts describing their boundaries. In order to obtain estimates about their second fundamental form $\B_{\Omega_m}$, such pointwise analysis needs to be extended to second-order derivatives, and this calls for the application of the implicit function theorem, for which  $\Om$ is required to be at least of class $C^2$.

This drawback is probably due to the fact that the above regularization  procedures are global in nature, i.e. they are obtained via mollification of functions ``globally" describing $\Om$, like its characteristic function or signed distance, whereas the second fundamental form of hypersurfaces of $\rn$ is defined via local parametrizations.

Comparatively, our proof relies on techniques which, in a sense, can be deemed as local in nature, since the starting point of our method is the regularization of the functions of $(n-1)$-variables which locally describe $\partial \Omega$. Thus, our approach seems more suitable when dealing with weak curvatures, though at the cost of  requiring  $ \Om$ to have a Lipschitz continuous boundary.  

\medskip

Regarding its applications,  approximation via a sequence of smooth bounded domains has proven to be a powerful tool especially when dealing with boundary value problems in Partial Differential Equations. Indeed, by tackling the same boundary value problem (or its suitable regularization) on smoother domains, accordingly one obtains smoother solutions, hence it is possible to perform all the desired computations and infer a priori estimates which do not depend on the full regularity of the approximating sets $\Om_m$, but only on their Lipschitz characteristics or other suitable quantities possibly depending on the second fundamental form $\B_{\Omega_m}$. 
 For instance, various investigations such as \cite{accfm,balci, cia19, Maz67,Maz73} showed that global regularity of solutions to linear and quasilinear PDEs may depend on a  weighted isocapacitary function for subsets $\partial \Omega$, the weight being the norm of the second fundamental form on $\partial \Omega$.

This function, which we denote by $\mathcal K_\Omega$, is defined as
\begin{equation}\label{dic7}
\mathcal K_\Omega(r) =
\displaystyle 
\sup _{
\begin{tiny}
 \begin{array}{c}{E\subset   B_r(x)} \\
x\in \partial \Omega
 \end{array}
  \end{tiny}
} \frac{\int _{\partial \Om \cap E} |\mathcal B_\Omega|d\mathcal H^{n-1}}{{\rm cap} (E, B_r(x))}\qquad \hbox{for $r>0$}\,,
\end{equation}
and it was first introduced in \cite{cia19}. Above,  ${\rm cap} (E, B_r(x))$ denotes the standard capacity of a compact set $E$ relative to the ball $B_r(x)$, i.e.
\begin{equation*}
    {\rm cap} (E, B_r(x))=\inf\bigg\{\int_{B_r(x)} |\nabla v|^2\,dx\,:\,v\in C_c^{0,1}(B_r(x))\,,\,v\geq 1\text{ on $E$}\bigg\}\,,
\end{equation*}
where $C^{0,1}_c(A)$ is the set of Lipschitz continuous functions with compact support in $A$.

We remark that, in order for $\K_\Omega(r)$ to be well defined, it suffices that $\partial \Omega$ is  Lipschitz continuous and belongs to $W^{2,1}$, as it can be inferred from inequalities \eqref{BBB} below.

\subsection*{Plan of the paper}
The rest of the paper is organized as follows: in Section \ref{sec:duue}, we  explain some non-standard notation used throughout the paper, and provide the definitions of $\mathcal{L}_\Omega$-Lipschitz domain, of $W^{2,q}$-domain and of weak curvature.

In Section \ref{sec:mainresult} we state in detail our main results, and we provide a few comments and an outline of their proofs.

 In Section \ref{secc:aux} we state and prove a useful convergence property of mollification and convolution, which will be used in the proof of the convergence properties of the approximating sets.

In Section \ref{sec:tras} we introduce the notion of \textit{transversality} of a unit vector $\mathbf{n}$ to a Lipschitz function $\phi$, and we show a very interesting fact, i.e. this transversality property is equivalent to the graphicality of $\phi$ with respect to the coordinate system $(y',y_n)$ having $\mathbf{n}=e_n$. We then close this section  by showing that the transversality condition-- hence the graphicality with respect to the reference system $(y',y_n)$-- is inherited by the convoluted function $M_m(\phi)$.

As a byproduct, we will find an interesting, yet expected result: if $\partial \Omega\in W^{2,q}$, then any Lipschitz function locally describing $\partial \Omega$ is of class $W^{2,q}$. This means that second-order Sobolev regularity is an intrinsic property  of the local charts describing $\partial \Omega$-- see Corollary \ref{yessa1}.

Finally, Section \ref{sec:apprmain} is devoted to the proof of the main Theorem \ref{appr:mthm}.

\section{Basic notation and definitions}\label{sec:duue}

In this section, we provide the relevant definitions and notation of use throughout the rest of the paper.

\begin{itemize}[leftmargin=*]
\itemsep3pt
\item For $d\in \N$, $U\subset \R^d$ open, and a function $v\,:U\,\to \R$, we shall denote by $\nabla v$ its $d$-dimensional gradient, and $\nabla^2 v$ its hessian matrix.
We will often use the short-hand notation for its level and sublevel sets
\begin{equation*}
    \begin{split}
        &\{v<0\}\coloneqq\{z\in U\,:\,v(z)<0\}.
        \\
        &\{v=0\}\coloneqq\{z\in U\,:\,v(z)=0\}\,.
    \end{split}
\end{equation*}

\item We denote by~$W^{k, p}(\Omega)$ the usual Sobolev space of~$L^p(\Omega)$ weakly differentiable functions having weak $k$-th order derivatives in~$L^p(\Omega)$. 

\noindent For any $\alpha\in (0,1]$, the spaces $C^k(\Omega)$ and  $C^{k,\alpha}(\Omega)$ will denote, respectively,  the space of functions with continuous and $\alpha$-H\"older continuous derivatives up to order $k\in \N$.

\item Point of $\R^n$ will be written as $x=(x',x_n)$, with $x'\in \R^{n-1}$ and $x_n\in \R$. We write $B_r(x)$ to denote the $n$-dimensional ball of radius $r>0$ and centered at $x\in \R^n$. Also, $B'_r(x')$ will denote  the $(n-1)$-dimensional ball of radius $r>0$ and centered at $x'\in \R^{n-1}$---when the centers are omitted, the balls are assumed to be centered at the origin, i.e. $B_r :=B_r(0)$ and $B_r' := B_r'(0')$.

\item For $d\in \N$, and for a given  matrix $X\in \mathbb{M}_{d\times d}$, we shall denote by $|X|$ its 
Frobenius Norm  $|X|=\sqrt{\mathrm{tr}(X^t X)}$, where $X^t$ is the transpose of $X$.

\item Given a Lebesgue measurable set $A$, we shall write $|A|$ for its Lebesgue measure.  Also, given two open bounded sets $A,B$, we will denote by $\mathrm{dist}_\H(A,B)$ their Hausdorff distance.

\item For a given function $\phi\,:\,U\to \R$ with $U\subset \R^{n-1}$ open, 
we write $G_\phi$ and $S_\phi$ to denote its graph and subgraph in $\R^n$, i.e.
\begin{equation*}
    G_\phi=\{x=\big(x',\phi(x')\big):\,x'\in U\}\quad\text{and}\quad  S_\phi=\{x=\big(x',x_n\big):\,x'\in U,\,x_n<\phi(x')\}\,.
\end{equation*}

\item We will denote by $\rho=\rho(x')$ the standard convolution Kernel in $\R^{n-1}$, i.e.
\begin{equation*}
    \rho(x')=
    \begin{dcases}
        \exp\bigg\{-\frac{1}{1-|x'|^2}\bigg\}\quad &\text{if $|x'|<1$}
        \\
        0\quad &\text{if $|x'|\geq 1$\,,}
    \end{dcases}
\end{equation*}
and we will write $\rho_m(x')=m^{n-1}\rho\big(m\,x'\big)$ for $m\in \N$.
Given $h\in L^1_{loc}(\R^{n-1})$, the convolution operator $M_m(h)$ is defined as
\begin{equation*}
    M_m(h)(x')=h\ast \rho_m(x')=\int_{\R^{n-1}}h(y')\,\rho_m(x'-y')\,dy'\,.
\end{equation*}

\end{itemize}

In the following, we specify the definition of Lipschitz domain and of Lipschitz characteristic.

\begin{definition}[Lipschitz characteristic of a  domain]\label{def:lip}
\rm{An open, connected set $\Om$ in $\rn$ is called a Lipschitz domain if 
 there exist  constants $L_\Om>0$ and $R_\Om \in (0, 1)$ 
such that, for every $x_0\in \partial \Om$ and $R\in (0, R_\Om]$ there exist an orthogonal coordinate system centered at $0\in\rn$ and an $L_\Om$-Lipschitz continuous function 
$\phi : B'_{R}\to (-\ell, \ell)$, where 
\begin{equation}\label{ell}
\ell = R (1+L_\Om),
\end{equation}
satisfying $\phi(0')=0$ , and
\begin{equation}\label{may100}
\begin{split}
    &\partial \Om \cap \big(B'_{R}\times (-\ell,\ell)\big)=\{(x', \phi (x'))\,:\,x'\in B'_{R}\},
    \\
    & \Om \cap \big(B'_{R}\times (-\ell,\ell)\big)=\{(x',x_n)\,:\,x'\in B'_{R}\,,\,-\ell<x_n<\phi (x')\}.
\end{split}
\end{equation}
Moreover, we set
\begin{equation}\label{may101}
\mathfrak L_\Om = (L_\Om, R_\Om),
\end{equation}
and call $\mathfrak L_\Om$ a Lipschitz characteristic of $\Om$.}
\end{definition}

It is easily seen that the above definition coincides with the standard one for uniformly Lipschitz domains--see e.g. \cite[Section 2.4]{henrot}. Our definition has the advantage of pointing out $\mathfrak{L}_\Om = (L_\Om, R_\Om)$ which appears in the characterization of our approximation sets.

We also remark that, in general, a  Lipschitz characteristic $\mathfrak{L}_\Om = (L_\Om, R_\Om)$ is not   uniquely determined. For instance, if $\partial \Omega \in C^1$, then  $L_\Om$ may be taken arbitrarily small, provided that  $R_\Om$ is chosen sufficiently small. 
\medskip

The function $\phi$ in definition \ref{def:lip} is typically called \textit{local (boundary) chart}. 
By Rademacher's theorem, this function is differentiable for $\H^{n-1}$-almost every $x'$, with gradient $\nabla\phi$ bounded by $L_\Om$. 
In particular, this implies that any Lipschitz domain  $\Omega$ admits a tangent plane on $\H^{n-1}$-almost every point of its boundary.

Moreover, the local chart $\phi$ naturally endows $\partial \Omega$ of a local parametrization  $\iota_\phi(x')=\big(x',\phi(x') \big)$, under which the first fundamental form $g=\{g_{ij}\}_{i,j=1}^{n-1}$ is given by
\begin{equation}\label{first:f}
    g_{ij}(x')=\delta_{ij}+\frac{\partial \phi(x')}{\partial x'_i}\,\frac{\partial \phi(x')}{\partial x'_j}\,,
\end{equation}
where $\delta_{ij}$ denotes the Kronecker's delta, and $x'$ is a point of differentiability of $\phi$. Then, the inverse matrix $g^{-1}=\{g^{ij}\}_{i,j=1}^{n-1}$ can be explictly computed:
\begin{equation}\label{inv:first}
    g^{ij}(x')=\delta_{ij}-\frac{1}{1+|\nabla \phi(x')|^2}\,\frac{\partial \phi(x')}{\partial x'_i}\,\frac{\partial \phi(x')}{\partial x'_j}\,.
\end{equation}

Since $\partial \Omega$ admits a tangent plane $\H^{n-1}$-almost everywhere, we may want to define a notion of weak second fundamental form, which extends the classical one for $C^\infty$-smooth domains of $\R^n$.
For this purpose, we need some additional regularity assumptions on $\phi$, and in particular on its second-order derivatives.

\begin{definition}[$W^{2,q}$ domains and weak curvature]\label{2q:domains}
  \rm{
Let $q\in [1,\infty)$.
  We say that a bounded Lipschitz domain $\Om$ is of class $W^{2,q}$ if the local boundary chart $\phi$ satisfying \eqref{may100} belongs to the Sobolev space $W^{2,q}(B'_R)$.
\\
 If $\phi\in W^{2,\infty}(B'_R)$, we say that $\partial \Omega\in C^{1,1}$ (or $\partial \Omega\in W^{2,\infty}$).
  
 If $\partial \Omega\in W^{2,1}$, the weak curvature $\B_\Omega=\{\B_{ij}\}_{i,j=1}^{n-1}$ of $\partial\Omega $ is locally defined as
\begin{equation}\label{def:B}
    \B_{ij}(x')=\frac{1}{\sqrt{1+|\nabla\phi(x')|^2}}\,\frac{\partial^2\phi(x')}{\partial x'_i\partial x'_j}\,,
\end{equation}
  for almost all points $x'$ of differentiability of $\phi$.
  Its norm is then given by
  \begin{equation}\label{deff:B}
      |\B_\Omega(x')|\coloneqq \frac{\sqrt{\mathrm{trace}\big((g^{-1}\,\nabla^2\phi)^2  \big)}}{\sqrt{1+|\nabla\phi(x')|^2}},
  \end{equation}
  where $g^{-1}$ is the inverse matrix of $g$ given by \eqref{inv:first}.}
\end{definition}

The reader may verify that identities \eqref{first:f}-\eqref{deff:B} concur with the usual ones when $\partial \Omega$ is a smooth hypersurface of $\rn$--see e.g. \cite[pp. 246-249]{lee2}.
However, these definitions also make sense when $\phi$ is merely Lipschitz continuous and belongs to the Sobolev space $W^{2,1}$. Indeed, the following inequalities hold true:
\begin{equation}\label{BBB}
   \frac{|\nabla^2\phi(x')|}{(1+L_\Om^2)^{3/2}}\leq |\B_\Omega(x')|\leq |\nabla^2\phi(x')|\,.
\end{equation}

In order to prove \eqref{BBB}, we first recall that for all symmetric matrices $X,Y$, with $X$ definite positive, we have the elementary linear algebra inequalities
\begin{equation*}
    \lambda^2_{\min}|Y|^2\leq \mathrm{tr}\big((XY)^2\big)\leq \lambda^2_{\max}\,|Y|^2\,,
\end{equation*}
where $\lambda_{\min},\lambda_{\max}$ denote the smallest and largest eigenvalues of $X$--see e.g. \cite[Lemma 3.6]{accfm} and its proof. 
Then, owing to \eqref{inv:first}, we observe that the largest and smallest eigenvalues of the matrix $g^{-1}$ are respectively $1$ and $(1+|\nabla \phi|^2)^{-1}$, and since $|\nabla \phi|\leq L_\Omega$ we immediately infer \eqref{BBB}. Inequalities \eqref{BBB} also show that (locally) second fundamental form $\B_\Omega$ is equivalent to the second-order derivatives of the local charts.

We close this section by pointing out that the above definitions can be easily extended to domains with boundary $\partial \Omega\in W^{k,q}$. Similarly, standard definitions follow for domains of class $C^k$ and $C^{k,\alpha}$.

\section{Main results}\label{sec:mainresult}
Having dispensed of the necessary definitions and notation, we can now give a precise statement of our main results. This is the content of this section, coupled with  a few comments and an outline of the proofs. Our first main result reads as follows.

\begin{theorem}\label{appr:mthm}
    Let $\Omega\subset \rn$ be a bounded, Lipschitz domain, with Lipschitz characteristic $\mathcal L_\Omega=(L_\Omega,R_\Omega)$.
    \\ (i) There exist sequences of bounded domains $\{\omega_m\},\{\Omega_m\}$, such that $\partial\omega_m\in C^\infty,\,\partial\Omega_m\in C^\infty$, and
    \begin{equation*}
        \omega_m\Subset\Omega\Subset \Omega_m\quad\text{for all $m\in \N$.}
    \end{equation*}
    Their diameters satisfy
    \begin{equation}\label{diameters}
    d_{\Omega_m}\leq c(n)\,d_\Omega\,,\quad d_{\omega_m}\leq c(n)\,d_\Omega\,,
\end{equation}   
the following convergence property hold true  \begin{equation}\label{leb:distance}
        \lim_{m\to\infty} |\Omega_m\setminus\Omega|=0\,,\quad \lim_{m\to\infty}|\Omega\setminus \omega_m|=0\,,
    \end{equation} 
  the Hausdorff distances safisfy
\begin{equation}\label{hauss:dist}
        \mathrm{dist}_\H(\omega_m,\Omega)+\mathrm{dist}_\H(\Omega_m,\Omega)\leq\frac{12\,L_\Omega\sqrt{1+L_\Omega^2}}{m}\quad\text{for all $m\in \N$,}
    \end{equation}
    and we may choose their Lipschitz characteristics $\mathcal{L}_{\Omega_m}=(L_{\Omega_m},R_{\Omega_m})$ and $\mathcal{L}_{\omega_m}=(L_{\omega_m},R_{\omega_m})$ such that
    \begin{equation}\label{lip:car}
    \begin{split}
        L_{\Omega_m}\leq c(n)(1+L_\Om^2)\,, & \quad   R_{\Omega_m}\geq R_\Omega/\big(c(n)(1+L_\Omega^2)\big) 
        \\
         L_{\omega_m}\leq c(n)(1+L_\Om^2)\,, &\quad  R_{\omega_m}\geq R_\Omega/\big(c(n)(1+L_\Omega^2)\big)\,,\quad\text{for all $m\in \N$.}
        \end{split}
    \end{equation}
    Moreover, the smooth boundaries $\partial \omega_m, \partial \Omega_m$ are described with the help of the same co-ordinate 
    systems as $\partial \Omega$, i.e. there exist finite number of local boundary charts $\{\phi^i\}_{i=1}^N,\{\psi^i_m\}_{i=1}^N$ and $\{\varphi^i_m\}_{i=1}^N$ which describe $\partial \Omega,\,\partial \Omega_m$ and $\partial \omega_m$ respectively, such that for each $i=1,\dots,N$ the functions $\psi^i_m,\varphi^i_m\in C^\infty$ are defined on the same reference system as $\phi^i$, and 
    \begin{equation}\label{lchart:conv1}
        \psi^i_m\xrightarrow{m\to\infty} \phi^i\quad\text{and}\quad \varphi^i_m \xrightarrow{m\to\infty} \phi^i\quad\text{in $ W^{1,p}(B'_{R_\Omega-\varepsilon_0})$}\,,
    \end{equation}
    for all $p\in [1,\infty)$, for all $i=1,\dots,N$, and any fixed constant $\varepsilon_0\in (0,R_\Omega/2)$.
\\ (ii)  If in addition $\partial \Omega\in W^{2,q}$ for some $q\in [1,\infty)$, then
\begin{equation}\label{lchart:conv2}
    \psi^i_m\xrightarrow{m\to\infty} \phi^i\quad\text{and}\quad \varphi^i_m \xrightarrow{m\to\infty} \phi^i\quad\text{in $W^{2,q}(B'_{R_\Omega-\varepsilon_0})$}\,,
\end{equation}
and there exists a constant $\widehat{c}=\widehat{c}(n,\mathcal{L}_\Omega, d_\Omega)$ such that
\begin{equation}\label{iscap:omm'}
        \K_{\Om_m}(r)+\K_{\omega_m}(r)\leq
        \begin{cases}
        \widehat{c}\,\Big\{ \K_{\Om}\big( \widehat{c}\,(r+\tfrac{1}{m})\big)+r\Big\}\quad & \text{if $n\geq 3$}
        \\
        \\
        \widehat{c}\,\Big\{ \K_{\Om}\big(\widehat{c}\,(r+\tfrac{1}{m})\big)+r\,\log(1+\tfrac{1}{r})\Big\}\quad & \text{if $n=2$}
        \end{cases}
    \end{equation}
    for all $m\in \N$ and $r\leq r_0(n,\mathcal{L}_\Omega)$.
\end{theorem}

Let us briefly comment on our result. Part (i) of Theorem \ref{appr:mthm} is mostly analogous to \cite[Theorem 5.1]{doktor}; as expected from domains $\Omega$  with Lipschitz continuous boundary, the local charts of $\partial \Omega_m,\partial \omega_m$ converge to the corresponding  local charts of $\partial \Omega$ in $W^{1,p}$ for all $p\in [1,\infty)$. In particular, by the classical Morrey-Sobolev's embedding Theorems, this entails an ``almost Lipschitz convergence",  i.e. the local charts $\psi^i_m$ and $\varphi^i_m$ converge to $\phi^i$ in every H\"older space $C^{0,\alpha}$ with $\alpha \in (0,1)$.

The main novelty of our result is given in Part (ii), where information about the second fundamental forms $\B_{\omega_m}$ and $\B_{\Omega_m}$ (or equivalently $\nabla^2\varphi^i_m$ and $\nabla^2 \psi^i_m$) is retrieved when $\partial \Omega$ is endowed with a weak curvature. For instance, by definition \eqref{def:B} and from the results of  Theorem \ref{appr:mthm}, via a standard covering argument it is easy to show that
\begin{equation} \label{lchart:conv3}  \int_{\partial\Omega_m}|\B_{\Omega_m}|^qd\H^{n-1}\to \int_{\partial \Omega}|\B_\Omega|^q d\H^{n-1}\quad\text{and}\quad \int_{\partial\omega_m}|\B_{\omega_m}|^q d\H^{n-1}\to \int_{\partial \Omega}|\B_\Omega|^qd\H^{n-1}\,,
\end{equation}
for all $q\in [1,\infty)$ such that $\partial \Omega\in W^{2,q}$.

Other than this, we obtain the isocapacitary estimate \eqref{iscap:omm'}, where
$\K_{\Omega}(r)$ and $\K_{\Omega_m},\K_{\omega_m}$ are the functions defined in \eqref{dic7} relative to $\Omega,\Omega_m$ and $\omega_m$, respectively. In the proof of \eqref{iscap:omm'}, we will also explicitly write the constant $\widehat{c}$ appearing therein.

Finally, the fixed parameter $\varepsilon_0\in (0,R_\Omega/2)$ appearing in \eqref{lchart:conv1} and \eqref{lchart:conv2} is purely technical, and does not affect the validity of the convergence results since  the boundaries $\partial \Omega,\,\partial \Omega_m $ and $\partial \omega_m$ all share the same coordinate cylinders of the kind $B'_{R_\Om/2}\times (-\ell,\ell)$, where $\ell=(1+L_\Omega)\,R_\Omega$.

\medskip

\noindent \textbf{Outline of the proof.} We fix a covering of $\partial \Omega$, with corresponding partition of unity $\{\xi_i\}_{i}$ and local boundary charts $\{\phi^i\}_{i}$, which are $L_\Omega$-Lipschitz continuous. 

Then we regularize each function $\phi^i$ via convolution, and add (or subtract) a suitable constant, so that we obtain $C^\infty$-smooth functions $\{\phi^i_m\}_{i}$ such that $\phi^i_m>\phi^i$ ( or $\phi^i_m<\phi^i$). 

However, in the original reference system, the graphs of these smooth functions $G_{\phi^i_m}$ are not ``glued" together, and thus their union is not the boundary of a domain, unlike the graphs $G_{\phi^i}$ whose union describes $\partial \Omega$-- see Figure 1 below.

To overcome this problem, we define a suitable $C^\infty$-smooth function $F_m$, built upon $\{\phi^i_m\}_i$ and $\{\xi_i\}_i$--  see equation \eqref{bdm:defin} below-- and define the regularized set $\Omega_m$ as the sublevel set $\{F_m<0\}$, so that 
\begin{equation*}
    \partial \Omega_m=\{F_m=0\},
\end{equation*}
and by construction we will have $\omega_m\Subset \Omega\Subset \Omega_m$.

\noindent  The function $F_m$ is called \textit{boundary defining functions} of $\Omega_m$-- see \cite[Section 5.4]{lee1}. 

In order to show that $\partial \Omega_m$ is a smooth manifold, we prove that the gradient of $F_m$ along the directions of graphicality of $\phi^i$ is greater than a positive constant depending on $L_\Omega$-- see estimate \eqref{coervic:Fm}.
This property of $F_m$ will be proven by exploiting the so-called \textit{transversality condition} of $\phi^i$,  which is inherited via convolution by $\phi^i_m$ as well. 
Therefore, $F_m$ is strictly monotone along these directions, which entails that its zero-level set $\partial \Omega_m$ is a smooth manifold with local boundary charts $\psi^i_m$ defined on the same reference system as $\phi^i$.

Thanks to the properties of convolution, we show that $F_m$ converge to the boundary defining function $F$ of $ \Omega$ built upon $\{\phi^i\}_{i}$ and $\{\xi_i\}_i$-- see equations \eqref{bd:defin} and \eqref{bdef:om}-- and thus $\psi^i_m$ converge uniformly to $\phi^i$.

Then, as in the proof of the implicit function theorem, we differentiate the identity $F_m\big(y',\psi^i_m(y')\big)=0$, so that we may express the gradient  $\nabla\psi^i_m$ (and its Hessian $\nabla^2 \psi^i_m$) in terms of $\{\phi^j_m, \nabla \phi^j_m\}_j$ (and $\{\nabla^2 \phi^j_m\}_{j}$), and then
 \eqref{lip:car}, \eqref{lchart:conv1} (and \eqref{lchart:conv2}) will be obtained by exploiting the convergence properties of convolution.
 
Finally, in order to get the isocapacitary estimate \eqref{iscap:omm'}, we make use of the estimates on $|\nabla^2 \psi^i_m|$ obtained in the previous steps, as to evaluate weighted Poincar\'e type quotients of the kind
\begin{equation*}
    \frac{\int_{\partial \Omega_m}v^2\,|\B_{\Omega_m}|\,d\H^{n-1}}{\int_{\R^n}|\nabla v|^2 dx}\,,\quad v\in C^{\infty}_c\big(B_r(x^0_m)\big),\,x^0_m\in \partial \Omega_m
\end{equation*}
in terms of the corresponding quotient with weight $|\B_\Omega|$, and then \eqref{iscap:omm'} will follow from the celebrated 
 isocapacitary equivalency Theorem of Maz'ya \cite{maz61}, \cite[Theorem 2.4.1]{maz}.

\medskip 

\begin{figure}[ht]
\centering
\includegraphics[width=0.8\textwidth]{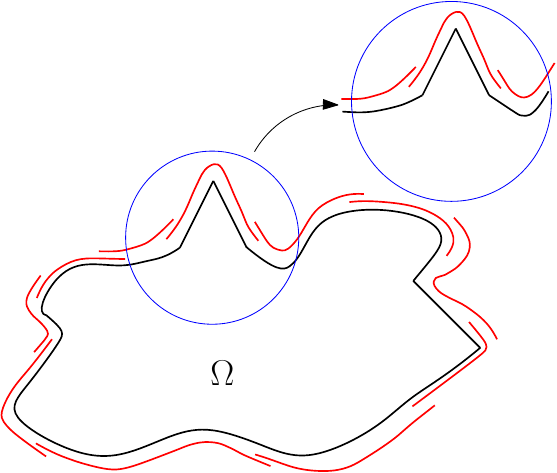}
\caption{In red: the graphs of the regularized local charts (up to isometry).}
\end{figure}

Our next and final result shows the flexibility of our approximation method, which takes into account even higher regularity of the domain $\Omega$. 

\begin{theorem}\label{appr:dd22}
    Under the same notations as Theorem \ref{appr:mthm}, we have that
    \begin{enumerate}
    \item if $\partial \Omega\in C^k$ for some $k\in\N$, then
    \begin{equation*}
        \psi^i_m\xrightarrow{m\to\infty} \phi^i\quad\text{and}\quad \varphi^i_m \xrightarrow{m\to\infty} \phi^i\quad\text{in $C^k(B'_{R_\Omega-\varepsilon_0})$};
    \end{equation*}
   \item  if $\partial \Omega\in C^{k,\alpha}$ for some $k\in\N$ and $\alpha\in (0,1)$, then
    \begin{equation*}
        \psi^i_m\xrightarrow{m\to\infty} \phi^i\quad\text{and}\quad \varphi^i_m \xrightarrow{m\to\infty} \phi^i\quad\text{in $C^{k,\alpha'}(B'_{R_\Omega-\varepsilon_0})$,}
    \end{equation*}
    for all $0<\alpha'<\alpha$;
    \item  if $\partial \Omega\in W^{k,q}$ for some $k\in\N$ and $q\in [1,\infty)$, then
    \begin{equation*}
        \psi^i_m\xrightarrow{m\to\infty} \phi^i\quad\text{and}\quad \varphi^i_m \xrightarrow{m\to\infty} \phi^i\quad\text{in $W^{k,q}(B'_{R_\Omega-\varepsilon_0})$.}
    \end{equation*}
    \item  if $\partial \Omega\in C^{k,1}$ for some $k\in\N$, then
    \begin{equation*}
        \psi^i_m\xrightarrow{m\to\infty} \phi^i\quad\text{and}\quad \varphi^i_m \xrightarrow{m\to\infty} \phi^i\quad\text{weakly-$\ast$ in $W^{k,\infty}(B'_{R_\Omega-\varepsilon_0})$.}
    \end{equation*}
\end{enumerate}

\end{theorem}

The proof of Theorem \ref{appr:dd22} can be easily carried out by extending the proof and estimates of Theorem \ref{appr:mthm} to higher order derivatives, and by using standard compactness theorems such as Ascoli-Arzel\'a's and weak-$\ast$ compactness. For this very reason, we decided to omit the proof.

\section{Auxiliary results}\label{secc:aux}

In this section, we state and prove a useful convergence property regarding the convolution of functions composed with a suitable family of bi-Lipschitz maps.

\begin{proposition}\label{prop:convoluz}
    \rm{Let $U\subset \R^{n-1}$ be a bounded domain, $K>0$ be a constant, and $\{\Psi_m\}_{m\in\N}$ be a family of bi-Lipschitz maps on $U$ such that 
    \begin{equation}\label{sup:psiem}
       \sup_{m\in\N}\|\nabla \Psi_m^{-1}\|_{L^\infty}\leq K \,,
    \end{equation}
 and there exists a bi-Lipschitz map $\Psi:U\to\Psi(U)$ such that
\begin{equation}\label{sium}
    \|\Psi_m-\Psi\|_{L^\infty(U)}\leq \frac{K}{m}\quad\text{for all $m\in\N$.}
\end{equation} 
Let $\mathcal{O}\subset \R^{n-1}$ open be such that $\Psi(U)\Subset \mathcal{O}$, and $\phi\in L^p(\mathcal{O})$ for some $p\in [1,\infty)$. Then
\begin{equation}\label{cia:cia}
 M_m(\phi)\circ \Psi_m\xrightarrow{m\to\infty} \phi\circ\Psi\quad\text{$\H^{n-1}$-a.e. in $U$ and in $L^p(U).$}  
\end{equation}
}
\end{proposition}

\begin{proof}
    Set
    \begin{equation*}
        U_\phi\coloneqq\big\{x'\in U\,:\, \Psi(x')\text{ is a Lebesgue point of } \phi\big\}
    \end{equation*}
  By Lebesgue differentiation theorem and since $\Psi$ is a bi-Lipschitz map, we have that $U_\phi$ is a subset of $U$ with full measure. Also, thanks to \eqref{sium} and the fact that $\Psi\big(U\big)\Subset \mathcal{O}$, we have that $\phi$ and $M_m(\phi)$ are well defined on a neighbourhood of $\Psi_m(U)$ for $m>m_0$ large enough. 
  Then, for all $x'\in U_\phi$ we have
  \begin{equation*}
      \begin{split} 
   \big|M_m(\phi)\big( \Psi_m (x')\big)  -\phi\big(\Psi(x')\big)\big|=\bigg|\int_{B'_{\frac 1 m}(\Psi_m(x'))}\Big[\phi(z')-\phi\big(\Psi(x')\big)\Big]\rho_m\big(\Psi_m(x')-z'\big)\,dz\bigg|
          \\
     \leq \big(\sup_{\R^{n-1}}\rho\big)\,m^{n-1}\int_{B'_{\frac{(K+1)}{m}}(\Psi(x'))}\big|\phi(z')-\phi\big(\Psi(x')\big) \big|\,dz'\xrightarrow{m\to\infty}0\,.
      \end{split}
  \end{equation*}
  Above we used the fact that $\Psi(x)$ is a Lebesgue point of $\phi$, and $B'_{\frac{1}{m}}(\Psi_m(x))\subset B'_{\frac{(K+1)}{m}}(\Psi(x))$ as a consequence of  \eqref{sium}.

  Now fix $\varepsilon>0$, and take a function $\widetilde{\phi} \in C^\infty_c(\R^{n-1})$ satisfying
  \begin{equation}\label{phi:epsi}
      \|\phi-\widetilde{\phi}\|^p_{L^p(\mathcal{O})}\leq \varepsilon\,.
  \end{equation}
  Standard properties of convolutions ensure that
  \begin{equation}\label{uuuu}
      \|M_m(\widetilde{\phi})-\widetilde{\phi}\|_{L^\infty(\mathcal{O})}\xrightarrow{m\to\infty}0\,.
  \end{equation}
Then we have
  \begin{equation}\label{trup}
      \begin{split}
          \int_U  & \big|M_m(\phi)\big( \Psi_m (x')\big)  -\phi\big(\Psi(x')\big)\big|^p\,dx'\leq c(p)\,\int_U\big| M_m(\phi-\widetilde{\phi})\big( \Psi_m (x')\big)\big|^p\,dx'
          \\
          & +c(p)\,\int_U \big|M_m(\widetilde{\phi})\big( \Psi_m (x')\big)- \widetilde{\phi}\big( \Psi(x')\big)\big|^p\,dx'+c(p)\,\int_U \big| \widetilde{\phi}\big( \Psi(x')-\phi\big( \Psi(x')\big)\big|^p\,dx'
      \end{split}
  \end{equation}
  By applying Jensen inequality, the change of variables $w'=\Psi_m (x')-z'$ and Fubini-Tonelli's Theorem we obtain
  \begin{equation*}
  \begin{split}
      \int_U\big| M_m(\phi-\widetilde{\phi})\big( \Psi_m (x')\big)\big|^p\,dx'\leq\int_U\int_{B'_{1/m}}\big| \phi\big( \Psi_m(x')-z'\big)-\widetilde{\phi}\big( \Psi_m(x')-z'\big)\big|^p \rho_m(z')\,dz'\,dx'
      \\
      \leq c(n)\,K^{n-1}\,\int_{\R^{n-1}}\rho_m(z')\,dz'\,\int_{\mathcal{O}}\big|\phi(w')-\widetilde{\phi}(w') \big|^p\,dw'\leq c(n)\,K^{n-1}\,\varepsilon\,,
      \end{split}
  \end{equation*}
  where we also used estimates \eqref{sup:psiem} and \eqref{phi:epsi}.

  Then, by using \eqref{sium} and \eqref{uuuu}, it is immediate to verify that
  \begin{equation*}
      \lim_{m\to\infty}\int_U \big|M_m(\widetilde{\phi})\big( \Psi_m (x')\big)- \widetilde{\phi}\big( \Psi(x')\big)\big|^p\,dx'= 0\,,
  \end{equation*}
  and finally, via a change of variables $y'=\Psi(x')$, and \eqref{phi:epsi} we get
  \begin{equation*}
      \int_U \big| \widetilde{\phi}\big( \Psi(x')-\phi\big( \Psi(x')\big)\big|^p\,dx'\leq  c(n)\,\|\nabla\Psi^{-1}\|_{L^\infty}^{n-1}\,\varepsilon\,.
  \end{equation*}
  Henceforth, by plugging the last three estimates into \eqref{trup}, we find
  \begin{equation*}
      \limsup_{m\to\infty}\int_U \big|M_m(\phi)\big( \Psi_m (x')\big)  -\phi\big(\Psi(x')\big)\big|^p\,dx'\leq c(n,p,L,\Psi)\,\varepsilon\,,
  \end{equation*}
  and thus \eqref{cia:cia} follows by the arbitrariness of $\varepsilon$.
\end{proof}

We close this section recalling a variant of Lebesgue dominated convergence Theorem which will be useful later on.

\begin{theorem}[Dominated convergence Theorem]\label{thm:dc}
    \rm{ Let $\{f_k\}_{k\in\N}$ be a sequence of measurable functions on $E\subset \R^{n-1}$ such that
\begin{enumerate}[label=(\roman*)]
    \item $f_k\to f$ almost everywhere on $E$;
    \item $|f_k|\leq g_k$ almost everywhere on $E$, with $g_k\in L^q(E)$ for some $q\in [1,\infty)$;
    \item there exists $g\in L^q(E)$ such that $g_k\to g$ a.e. on $E$, and $\int_E g_k^q\,dx\to \int_E g^q\,dx$.
\end{enumerate}   
Then $f\in L^q(E)$, and
\begin{equation*}
    \int_E |f_k-f|^q\,dx\to 0\,.
\end{equation*}
    }
\end{theorem}

\section{Transversality and graphicality}\label{sec:tras}

Throughout this section, we shall consider an isometry $T$ of $\R^n$, such that
\begin{equation}\label{isommm}
    Tx=\mathcal{R}x+x^0\,,\quad x\in \R^n\,,
\end{equation}
where $\mathcal{R}=\big\{\mathcal{R}_{ij}\big\}_{i,j=1}^n$ is an orthogonal matrix of $\R^n$, and $x^0\in \R^n$. Let 
\begin{equation*}
    \mathbf{n}=\mathcal{R}^t e_n\in \mathbb{S}^{n-1}\,,
\end{equation*}
 where $e_n$ denotes the $n$-th canonical vector of $\R^n$, i.e. $e_n=(0,\dots,0,1)$, $\mathcal{R}^t$ is the transpose matrix of $\mathcal{R}$, and $\mathbb{S}^{n-1}$ is the unit sphere on $\R^n$. \vspace{0.3cm}

Here we introduce the geometric notion of \textit{transversality}, which was already used in \cite{hofmann} in a wider sense. The definition given here suffices to our purposes.

\begin{definition}[Transversality] \label{def:transversal}\rm{ Let $\phi:U\to\R$ be a Lipschitz continuous function on $U\subset \R^{n-1}$ open. We say that a unit vector $\mathbf{n}\in \mathbb{S}^{n-1}$ is transversal to $\phi$ if there exists $\kappa>0$ such that}
\begin{equation*}
    \mathbf{n}\cdot\nu(x')\geq \kappa\quad\text{for $\H^{n-1}$-a.e. $x'\in U$},
\end{equation*}
    where $\nu$ denotes the outward normal to $G_\phi$ with respect to the subgraph $S_\phi$.
\end{definition}

The next proposition shows a very interesting feature: the transversality of  $\mathbf{n}\in \mathbb{S}^{n-1}$ to a Lipschitz  function $\phi$  is equivalent to 
the graphicality (and subgraphicality) of $\phi$ with respect to any  reference system having $e_n=\mathbf{n}$, that is  after performing a rotation of the axes through $\mathcal{R}$, the graph and subgraph of $\phi$ are mapped onto the graph and subgraph of another function $\psi$-- see identities \eqref{rotaz:z} below.

\begin{proposition}\label{prop:trass}
    \rm{
Let $U\subset \R^{n-1}$ be open, $\phi:\,U\to \R$ be a Lipschitz function, let $T$ be an isometry of the form \eqref{isommm}, and let $\mathbf n=  \mathcal{R}^te_n$. 
\\ (i) If there exists an $L$-Lipschitz function $\psi:V\to\R$  such that
\begin{equation}\label{rotaz:z}
 TG_\phi=G_{\psi}\quad\text{and}\quad 
TS_\phi=S_{\psi}\cap T(U\times \R)\,,
\end{equation}
then we have the transversality condition
\begin{equation}\label{quant:trans}
   \mathbf n\cdot\nu(x')\geq\frac{1}{\sqrt{1+L^2}}\quad\text{for $\H^{n-1}$-a.e. $x'\in U$.}
\end{equation}
(ii) Viceversa, if $\phi \in C^k(U)$ for some $k\in \N$ and \eqref{quant:trans} holds, then there exist $V\subset \R^{n-1}$ open, and a function $\psi\in C^k(V)$ such that $\|\nabla \psi\|_{L^\infty(V)}\leq L$ and \eqref{rotaz:z} holds true.
    }
\end{proposition}

Let us comment on this result. Part (i) states that if $G_\phi$ and $S_\phi$ are, respectively, the graph and subgraph of an $L$-Lipschitz function $\psi$ with respect to the reference system $z=(z',z_n)$ having $\mathbf{n}=e_n$, then the quantitative transversality estimate \eqref{quant:trans} holds true.

Part (ii) states the opposite in the $C^k$ case: the transversality condition \eqref{quant:trans} implies the graphicality and subgraphicality of $\phi$ with respect to the coordinate system $z=(z',z_n)$, and it also provides a Lipschitz estimate to $\psi$.

\vspace{0.4cm}

\noindent Before starting the proof, we need to introduce the so-called \textit{transition map} $\C$ from $\phi$ to $\psi$.
Under the same notation as Proposition \ref{prop:trass}, the transition map $\C:U\to V$ is defined as
\begin{equation*}
    \C x'\coloneqq\Pi\,T\big(x',\phi(x')\big)\,.
\end{equation*}
Here $\Pi:\R^n\to \R^{n-1}$ is the projection map $\Pi(x',x_n)=x'$. 
Observe that, when identities \eqref{rotaz:z} hold true, by the very definition of $\C$  we have the equation
\begin{equation*}
    T\big(x',\phi(x')\big)=\big(\C x',\psi\big(\C x'\big)\big)
\end{equation*}
In particular, this implies that $\C$ is a bijection, with inverse function $\C^{-1}:V\to U$ given by 
\begin{equation*}
    \C^{-1}z'=\Pi\,T^{-1}\big(z',\psi(z')\big)\,.
\end{equation*}

Also, since $\phi,\psi$ are Lipschitz continuous, then $\C$ is a bi-Lipschitz tranformation from $U$ to $V$.

\begin{figure}[h]
\centering
\includegraphics[width=0.8\textwidth]{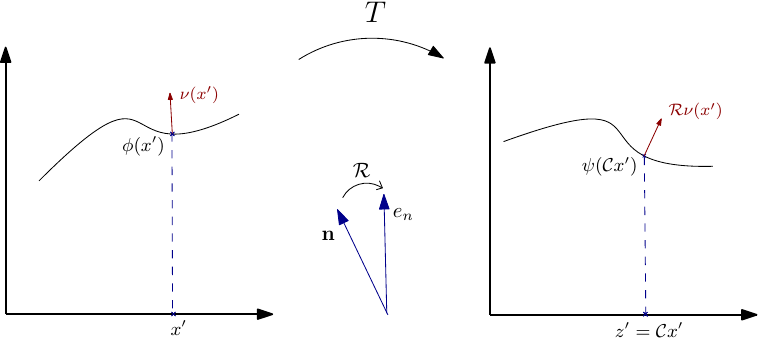}
\caption{}
\end{figure}

\begin{proof}[Proof of Proposition \ref{prop:trass}](i) By Rademacher's theorem, the normal vector $\nu$ to $G_\phi$ outward with respect to $S_\phi$ is well defined $\H^{n-1}$-almost everywhere, and thanks to \eqref{rotaz:z} and the definition of $\C$, we may write
\begin{equation}\label{normali}
    \nu(x')=\frac{(-\nabla \phi(x'),1)}{\sqrt{1+|\nabla\phi(x')|^2}}=\mathcal{R}^t\Bigg(\frac{(-\nabla \psi(\C x'),1)}{\sqrt{1+|\nabla\psi(\C x')|^2}}\Bigg)\quad\text{$\H^{n-1}$-a.e. $x'\in U$.}
\end{equation}
Therefore, since $\mathcal{R}\mathbf{n}=e_n$ and $|\nabla\psi|\leq L$, from \eqref{normali} we infer
\begin{equation}\label{normali:cdot}
    \mathbf{n}\cdot\nu(x')=e_n\cdot\mathcal{R}\nu(x')=\frac{1}{\sqrt{1+|\nabla\psi(\C x')|^2}}\geq\frac{1}{\sqrt{1+L^2}}\quad\text{for $\H^{n-1}$-a.e. $x'\in U$.}
\end{equation}
(ii) Assume $\phi\in C^k(U)$ and that \eqref{quant:trans} is in force. 

Consider the $C^k$-function $f:U\times \R\to \R$, defined as $f(x)\coloneqq x_n-\phi(x')$, so that
\begin{equation}\label{grafo:tempp}
    \{f=0\}=G_\phi\quad\text{and}\quad \{f<0\}=S_\phi\,.
\end{equation}

Now let $\tilde{f}:T(U\times \R)\to \R$ be the function defined as $ \tilde{f}(z)=f(x)$ for $z=Tx$.
Recalling $\mathcal{R}\mathbf{n}=e_n$, via the chain rule we compute
\begin{equation}\label{belll}
    \frac{\partial \tilde{f}(z)}{\partial z_n}=\mathcal{R}_{nn}-\sum_{k=1}^{n-1}\frac{\partial \phi(x')}{\partial x'_k}\,\mathcal{R}_{nk}=(-\nabla \phi(x'),1)\cdot \mathbf{n}\,.
\end{equation}
Thus, from expression \eqref{normali} of $\nu(x')$ and estimate \eqref{quant:trans}, we obtain
\begin{equation}\label{coer:ftilde}
    \frac{\partial \tilde{f}(z)}{\partial z_n}=\sqrt{1+|\nabla\phi(x')|^2}\,\nu(x')\cdot\mathbf{n}\geq\frac{1}{\sqrt{1+L^2}}\quad\text{for $z=Tx$.}
\end{equation}
Therefore, owing to \eqref{coer:ftilde} and the implicit function theorem, we immediately infer the existence  of a function $\psi\in C^k(V)$, with $V\subset \R^{n-1}$ open, such that
\begin{equation*}
    \{\tilde{f}=0\}=G_\psi\quad\text{and}\quad  \{\tilde{f}=0\}=S_\psi\cap T(U\times \R).
\end{equation*}
Thereby, \eqref{rotaz:z} follows from the very definition of $\tilde{f}$ and \eqref{grafo:tempp}.
 
Finally, by using \eqref{normali:cdot} we infer that 
$|\nabla\psi(\C x')|\leq L$ for all $x'\in U$, whence $\|\nabla \psi\|_{L^\infty(V)}\leq L$ since the transition map $\C$ is a bijection.

\end{proof}

\begin{remark}\label{bel:remarkkk}\rm{
    We point out that inequality \eqref{coer:ftilde}, when evaluated  at points $z=T\big(x',\phi(x')\big)$, holds true if $\phi$ and $\psi$ are merely Lipschitz continuous and satisfy \eqref{rotaz:z}.
    
    Indeed, since $\C$ is a bi-Lipschitz map, by Rademacher's Theorem and the chain rule we may perform the same computations as \eqref{belll}-\eqref{coer:ftilde} and get
    \begin{equation}\label{coer:ff}
       \mathcal{R}_{nn}-\sum_{k=1}^{n-1}\frac{\partial \phi(x')}{\partial x'_k}\,\mathcal{R}_{nk}\geq\frac{1}{\sqrt{1+L^2}}\quad\text{for $\H^{n-1}$-a.e. $x'\in U$.}
    \end{equation}
    }
\end{remark}

By making use of this information, we now show that the transversality condition \eqref{quant:trans} is inherited by the regularized function $M_m(\phi)$. This is the content of the following proposition

\begin{proposition}\label{pr:rotconv}
    \rm{Let $U,\,V\subset \R^{n-1}$ be open bounded , let $T$ be an isometry of the form \eqref{isommm}, and $\mathbf n=\mathcal{R}^t e_n$. Let $\phi:\,U\to \R$ and $\psi:\,V\to\R$ be $L$-Lipschitz functions satisfying \eqref{rotaz:z}. If we set
    \begin{equation*}
        U_m\coloneqq\big\{x'\in U\,:\,\mathrm{dist}(x',\partial U)>\tfrac{1}{m}\big\}
    \end{equation*}
and for some sequence $\{c_m\}_{m\in \N}\subset \R$ we define
\begin{equation*}
    \phi_m(x')\coloneqq M_m(\phi)(x')+c_m\quad\text{for $x'\in U_m$,} 
\end{equation*}
then $\phi_m$ is $L$-Lipschitz continuous on $U_m$ and
\begin{equation}\label{Lm:conv}
    \|\phi_m-\phi\|_{L^\infty(U_m)}\leq \frac{L}{m}+|c_m|\,.
\end{equation}
In addition, we have the transversality condition
\begin{equation}\label{quasi:transM}
    \mathcal{R}_{nn}-\sum_{k=1}^{n-1}\frac{\partial \phi_m}{\partial x'_k}(x')\mathcal{R}_{nk}=\big(-\nabla \phi_m(x'),1\big)\cdot\mathbf n\geq\frac{1}{\sqrt{1+L^2}}\quad\text{for all $x'\in U_m$,}
\end{equation}
and  
\begin{equation}\label{quant:transM}
    \mathbf{n}\cdot \nu_m(x')\geq \frac{1}{1+L^2}\quad\text{for all $x'\in U_m$\,,}
\end{equation}
where $\nu_m$ is the outward unit normal to $G_{\phi_m}$ with respect to the subgraph $S_{\phi_m}$.}
\end{proposition}

\begin{proof}
   Let $x'_0\in U_m$. By multiplying \eqref{coer:ff} with $\rho_m(x'_0-x')$ and integrating in $x'$ we immediately obtain 
   \begin{equation*}
       \mathcal{R}_{nn}-\sum_{k=1}^{n-1}\frac{\partial M_m(\phi)(x_0')}{\partial x'_k}\,\mathcal{R}_{nk}\geq\frac{1}{\sqrt{1+L^2}}\quad\text{for all $x'_0 \in U_m$,}
   \end{equation*}
   and  \eqref{quasi:transM} holds true.

Next, from the $L$-Lipschitz continuity of $\phi$, we have
\begin{equation*}
\begin{split}
    \big|M_m(\phi)(x')-M_m(\phi)(y')\big| & \leq\int_{\R^{n-1}}\big|\phi(x'-z')-\phi(y'-z')\big|\,\rho_m(z')\,dz'
    \\
    &\leq L\,|x'-y'|\,\int_{\R^{n-1}}\rho_m(z')\,dz'=L\,|x'-y'|
    \end{split}
\end{equation*}
for all $x',y'\in U_m$, 
 hence $\phi_m$ is $L$-Lipschitz continuous as well. From this and \eqref{quasi:transM}, we get
\begin{equation*}
    \mathbf{n}\cdot \nu_m(x')=\mathbf{n}\cdot\frac{\big(-\nabla M_m(\phi)(x'),1\big)}{\sqrt{1+|\nabla M_m(\phi)(x')|^2}}\geq\frac{1}{1+L^2}\quad\text{for all $x' \in U_m$,}
\end{equation*}
that is \eqref{quant:transM}. Next, since $\rho_m$ is radially symmetric and $\phi$ is $L$-Lipschitz continuous, for all $x'\in U_m$ we get 
\begin{align*}
    \big|M_m(\phi)(x')-\phi(x')\big| & \leq\int_{B'_{1/m}}\big| \phi(x'+y')-\phi(x')\big|\,\rho_m(y')\,dy'
    \\
    & \leq \int_{B'_{1/m}}L\,|y'|\,\rho_m(y')\,dy'\leq \frac{L}{m}\,,
\end{align*}
and thus \eqref{Lm:conv} follows.
\end{proof}

Since we have proven that the regularized function $M_m(\phi)$ satisfies the transversality condition, Part (ii) of Proposition \ref{prop:trass} entails its ``graphicality" with respect to the coordinate system having $\mathbf{n}=e_n$.

\begin{proposition}\label{pr:rotconv1}\rm{
    Under the same assumptions of Proposition \ref{pr:rotconv}, there exist $V_m\subset \R^{n-1}$ open bounded such that 
\begin{equation}\label{VmquasiV}
    \mathrm{dist}_\H(V_m,V)\leq \frac{2\sqrt{1+L^2}}{m}+|c_m|\,,
\end{equation}
 and a function $\psi_m\in C^\infty(V_m)$  satisfying 
 \begin{equation}\label{nbvmm}
     \|\nabla \psi_m\|_{L^\infty(V_m)}\leq 2(1+L^2)\,,
 \end{equation}
\begin{equation}\label{rotaz:zM}
 TG_{\phi_m}=G_{\psi_m}\quad\text{and}\quad 
TS_{\phi_m}=S_{\psi_m}\cap T\big(U_m\times \R\big)\,.
\end{equation}
If in addition $V_m\cap V\neq \emptyset$, then 
\begin{equation}\label{c:convrot}
    \|\psi_m-\psi\|_{L^\infty(V_m\cap V)}\leq \frac{L(1+L)}{m}+(1+L)\,|c_m|,
\end{equation}
and if  $\C_m$ is the transition map of $\phi_m$, we have that
\begin{equation}\label{uoaaa0}
    \|\C_m-\C\|_{L^\infty(U_m)}+\|\C^{-1}_m-\C^{-1}\|_{L^\infty(V_m\cap V)}\leq c(n)\,(1+L^2)\Big(\frac{1}{m}+|c_m|\Big)\,.
\end{equation}
}
\end{proposition}

\begin{proof}
    From the results of Part (ii) of Proposition \ref{prop:trass} and \eqref{quant:transM}, there exist $V_m\subset \R^{n-1}$ open bounded, and a function $\psi_m\in C^\infty(V_m)$ such that \eqref{rotaz:zM} holds. Also, owing to \eqref{quant:trans}, we immediately obtain \eqref{nbvmm}.

Now we recall that the transition map of $\phi_m$ is  the function $\C_m:\, U_m\to V_m$ defined as $\C_m x'=\Pi\,T\big(x',\phi_m(x')\big)$, and for all $x'\in U_m$ we have
\begin{equation*}
T\big(x',\phi(x')\big)=\big(  \C x',\psi(\C x')\big)\quad\text{and}\quad T\big(x',\phi_m(x')\big)=\big(  C_m x',\psi_m(\C_m x')\big)\,,
\end{equation*}
so that from \eqref{Lm:conv} we infer
\begin{equation*}
    \begin{split}
       |c_m|+ \frac{L}{m}\geq |\phi_m(x')-\phi(x')|=\big|\big(x',\phi_m(x')\big)-\big(x',\phi(x')\big)\big|= \big|\big(\C_m x',\psi_m(\C_m x')\big)-\big(\C x',\psi(\C x') \big)\big|\,,
    \end{split}
\end{equation*}
for all $x'\in U_m$. In particular
\begin{equation}\label{uoaaa1}
    \begin{dcases}
        |\C_m x'-\C x'|\leq \frac{L}{m}+|c_m|
        \\
        \big|\psi_m\big(\C_m x'\big)-\psi\big(\C x'\big)|\leq \frac{L}{m}+|c_m|
    \end{dcases}
    \quad\text{for all $x'\in U_m$}
\end{equation}
The first inequality in \eqref{uoaaa1} entails $\mathrm{dist}_\H\big(V_m, \C (U_m)\big)\leq \frac L m+|c_m|$.

On the other hand,  by definition of $U_m$, for any $x'\in U$ we may find $x'_m\in U_m$ such that $|x'-x'_m|\leq \frac{1}{m}$. Since $\Pi$ and $T$ are $1$-Lipschitz continuous, and $\phi$ is $L$-Lipschitz continuous, it follows that
\begin{equation*}
    |\C x'-\C x'_m|\leq \big|\big( x',\phi(x')\big)-\big( x'_m,\phi(x'_m)\big)\big|\leq \frac{\sqrt{1+L^2}}{m}\,,
\end{equation*}
 which implies $\mathrm{dist}_\H\big(\C (U_m),V\big)\leq \frac{\sqrt{1+L^2}}{m}$ since $\C(U)=V$. Hence, by using the triangle inequality we get
 \begin{equation*}
     \mathrm{dist}_\H\big(V_m,V\big)\leq \mathrm{dist}_\H\big(V_m,\C (U_m)\big)+\mathrm{dist}_\H\big(\C (U_m),V\big)\leq \frac{2\sqrt{1+L^2}}{m}+|c_m|\,,
 \end{equation*}
 that is \eqref{VmquasiV}. 
 
 Next, on assuming that $V_m\cap V\neq \emptyset$, and $\C_m$ being a bijection between $U_m$ and $V_m$, we may take a point $y'\in V_m\cap V$ such that $y'=\C_m x'$ for some $x'\in U_m$ From \eqref{uoaaa1} we find
 \begin{equation*}
     |\C_m x'-\C x'|=|y'-\C\C_m^{-1} y'|\leq \frac{L}{m}+|c_m|\,,
 \end{equation*}
and
\begin{equation*}
    \big|\psi\big( \C x'\big)-\psi_m\big(\C_m x'\big)\big|=\big| \psi\big( \C\C^{-1}_m y'\big)-\psi_m(y')\big|\leq \frac{L}{m}+|c_m|\,.
\end{equation*}
By using these two estimates and the $L$-Lipschitz continuity of $\psi$, we obtain
\begin{equation*}
    \begin{split}
        |\psi(y') & -\psi_m(y')|  \leq |\psi(y')-\psi(\C\C_m^{-1}y')|+ |\psi(\C\C_m^{-1}y')-\psi_m(y')|
        \\
        & \leq L\,|y'-\C\C_m^{-1} y'|+\frac{L}{m}+|c_m|\leq \frac{L(1+L)}{m}+(1+L)\,|c_m|\quad\text{for all $y'\in V_m\cap V$,}
    \end{split}
\end{equation*}
that is \eqref{c:convrot}. Finally, by making use of \eqref{c:convrot} and a similar argument as in the proof of \eqref{uoaaa1}, we obtain \eqref{uoaaa0}.
\end{proof}

The next proposition shows that if $\phi\in W^{2,q}$, then $\psi\in W^{2,q}$ as well. Namely, graphicality preserves Sobolev second-order regularity for Lipschitz functions.

\begin{proposition}\label{yessa}
    Under the same assumptions of Propositions \ref{pr:rotconv}-\ref{pr:rotconv1}, if in addition $\phi\in W^{2,q}_{loc}(U)$ for some $q\in [1,\infty]$, then $\psi\in W^{2,q}_{loc}(V)$.
\end{proposition}

\begin{proof}
   In the following proof, we will make use of Propositions \ref{pr:rotconv}-\ref{pr:rotconv1} with $c_m\equiv 0$.
    
    Fix $U_0\Subset U$ open, and set $V_0=\C (U_0)$. Since $\mathrm{dist}_\H(V_m,V)\to 0$ due to \eqref{VmquasiV}, from \cite[Proposition 2.2.17]{henrot} we may find $m_0>0$ large enough such that
    \begin{equation*}
        V_0\Subset V\cap V_m \quad\text{for all $m>m_0$.}
    \end{equation*}
    Now let 
    \begin{equation*}
        f_m(x)=x_n-M_m(\phi)(x')\quad\text{for $x\in U_m\times \R$,}
    \end{equation*}
and set $\widetilde{f}_m(y)\equiv f_m(x)$ for $y=Tx$. 
Then owing to \eqref{rotaz:zM}, we have that $\widetilde{f}_m\big(y',\psi_m(y')\big)=0$ for all $y'\in V_m$. By differentiating this expression, we obtain
    \begin{equation}\label{oct1}
        \frac{\partial \psi_m}{\partial y'_k}(y')=-\bigg(\frac{\partial\widetilde{f}_m}{\partial y_n}\big(y',\psi_m(y') \big)\bigg)^{-1}\bigg(\frac{\partial \widetilde{f}_m}{\partial y'_k}\big(y',\psi_m(y')\big)\bigg)\,,
    \end{equation}
    and from the chain rule, equation $\mathbf{n}=\mathcal{R}^t e_n$, the definition of $\C_m^{-1}$ and \eqref{quasi:transM}, we have
\begin{equation}\label{oct2}
    \begin{split}\frac{\partial\widetilde{f}_m}{\partial y'_k}\big(y',\psi_m(y') \big) & =\mathcal{R}_{kn}-\sum_{l=1}^{n-1}\frac{\partial M_m(\phi)}{\partial x'_l}(\C^{-1}_m y')\,\mathcal{R}_{kl}
\\
\frac{\partial\widetilde{f}_m}{\partial y_n}\big(y',\psi_m(y') \big) & =\mathcal{R}_{nn}-\sum_{l=1}^{n-1}\frac{\partial M_m(\phi)}{\partial x'_l}(\C^{-1}_m y')\mathcal{R}_{nl}\geq \frac{1}{\sqrt{1+L^2}}\,,
\end{split}
\end{equation}

Moreover, thanks to 
\eqref{nbvmm} and the $L$-Lipschitz continuity of $M_m(\phi)$, the maps $\C_m$ are uniformly bi-Lipschitz, i.e.
\begin{equation*}
    \|\nabla\C_m\|_{L^\infty}+\|\nabla\C_m^{-1}\|_{L^\infty}\leq C(n,L).
\end{equation*}
Thanks to this piece of information and \eqref{uoaaa0}, we may apply Proposition \ref{prop:convoluz} and get
\begin{equation}\label{oct3}
    \nabla M_m(\phi)(\C^{-1}_m y')\to   \nabla\phi(\C^{-1} y')\quad\text{for $\H^{n-1}$-a.e. $y'\in V_0$}
\end{equation}
By combining \eqref{oct1}-\eqref{oct3}, and by using dominated convergence theorem, we find that $\nabla \psi_m$ converges in $L^p(V_0)$ to some vector-valued function $G$ for all $p\in[1,\infty)$. It then follows from \eqref{c:convrot} and the uniqueness of the distributional limit that $G=\nabla \psi$, hence
\begin{equation}\label{oct4}
    \nabla\psi_m\to \nabla \psi\quad\text{$\H^{n-1}$-a.e. in $V_0$ and in $L^p(V_0)$. }
\end{equation}
Next, we differentiate twice identity $\widetilde{f}_m\big(y',\psi_m(y')\big)=0$, and for $k,r=1,\dots,n-1$ we obtain
\begin{equation}\label{oct5}
\begin{split}
    \frac{\partial^2 \psi_m}{\partial y'_k \partial y'_r}(y')=-\bigg(\frac{\partial\widetilde{f}}{\partial y_n}\big(y',\psi_m(y')\big) \bigg)^{-1}\bigg\{ 
 & \frac{\partial^2 \widetilde{f}}{\partial y'_k \partial y'_r}\big( 
 y',\psi_m(y')\big)+\frac{\partial^2 \widetilde{f}}{\partial y'_k \partial y_n}\big( 
 y',\psi_m(y')\big)\,\frac{\partial \psi_m}{\partial y'_r}(y')
 \\
 & +\frac{\partial^2 \widetilde{f}}{\partial y'_r \partial y_n}\big( 
y',\psi_m(y')\big)\,\frac{\partial \psi_m}{\partial y'_k}(y')
\\
& +\frac{\partial^2 \widetilde{f}}{\partial y_n \partial y_n}\big( 
 y',\psi_m(y')\big)\,\frac{\partial \psi_m}{\partial y'_k}(y')\,\frac{\partial \psi_m}{\partial y'_r}(y')\bigg\}\,,
\end{split}
\end{equation}
while from the chain rule and the properties of $\C_m$, we obtain
\begin{equation}\label{oct6}
    \frac{\partial^2 \widetilde{f}}{\partial y'_k \partial y'_r}\big( 
 y',\psi_m(y')\big)=-\sum_{l,t=1}^{n-1}\frac{\partial^2 M_m(\phi)}{\partial x'_l\partial x'_t}(\C^{-1}_m y')\,\mathcal{R}_{kl}\mathcal{R}_{rt}\,.
\end{equation}
Then, another application of Proposition \ref{prop:convoluz} entails that 
\begin{equation*}
    \nabla^2 M_m(\phi)(\C^{-1}_m y')\to \nabla^2 \phi(\C^{-1} y')\quad\text{for $\H^{n-1}$-a.e. $y'\in V_0$ and in $L^q(V_0)$,}
\end{equation*}
in the Case $q\in [1,\infty)$. From this, \eqref{oct2}, \eqref{oct4}-\eqref{oct6} and by using dominated convegence Theorem \ref{thm:dc}, we find that $\nabla^2 \psi_m$ converges in $L^q(V_0)$ to some matrix valued function $H$. Whence $H=\nabla^2 \psi$ due to the uniqueness of the distributional limit, and the proof in the Case $q\in [1,\infty)$ is complete due to the arbitrariness of $U_0$.

In the Case $q=\infty$, from  \eqref{oct2}, \eqref{oct5} and \eqref{oct6} we infer that $\{\psi_m\}_m$ is a sequence uniformly bounded in $W^{2,\infty}(V_0)$ with respect to $m$. Therefore, up to a subsequence, we have that $\psi_m$ weakly-$\ast$ converge in $W^{2,\infty}(V_0)$ to $\psi$, thus completing the proof.
\end{proof}

At last, we close this section with the following intrinsic  property of $W^{2,q}$ domains.

\begin{corollary}\label{yessa1}
    Let $\Omega$ be a bounded Lipschitz domains such that $\partial \Omega\in W^{2,q}$ for some $q\in [1,\infty]$. Then any Lipschitz local chart $\psi$ of $\partial \Omega$ is of class $W^{2,q}$.
\end{corollary}
\begin{proof}
    From Definition \ref{2q:domains}, there exists a Lipschitz local chart $\phi\in W^{2,q}$ and an isometry $T$ such that \eqref{rotaz:z} holds. The thesis then follows from Proposition \ref{yessa}.
\end{proof}
As a final remark, let us mention that both Proposition \ref{yessa} and Corollary \ref{yessa1} can be easily extended to the $W^{k,q}$ Case.

\section{Proof of Theorem \ref{appr:mthm}}\label{sec:apprmain}
This section is devoted to the proof of Theorem \ref{appr:mthm}, which is divided into a few steps.

\noindent From here onward, $m_0$ and $k_0$ will denote  positive integers, possibly changing from line to line.

\subsection{Covering of \texorpdfstring{$\partial \Omega$}{}}
By Definition \ref{def:lip}, for any $x_0\in \partial \Omega$, we may find an $L_\Omega$-Lipschitz function $\phi^{x_0}:B'_{R_\Omega}\to \R$, and an isometry $T^{x_0}$ of $\rn$ such that $T^{x_0}x_0=0$, and
\begin{equation*}
\begin{split}
    & T^{x_0} \partial \Om \cap \big(B'_{R_\Omega}\times (-\ell,\ell)\big)=\big\{(y', \phi^{x_0} (y'))\,:\,y'\in B'_{R_\Omega}\big\},
    \\
    & T^{x_0} \Om \cap \big(B'_{R_\Omega}\times (-\ell,\ell)\big)=\big\{(y',y_n)\,:\,x'\in B'_{R_\Omega}\,,\,-\ell<y_n<\phi^{x_0} (y')\big\}\,,
\end{split}
\end{equation*}
where $\ell=R_\Omega(1+L_\Omega)$.
Let us consider the open covering $\{B_{R_\Omega/8}(x_0)\}_{x_0\in\partial \Omega}$ of $\partial \Omega$
\footnote{Any other open covering is allowed, as long as its sets are strictly contained in the coordinate cylinders $B'_{R_\Omega}\times (-\ell,\ell)$. The open covering here chosen helps simplifying a few computations, especially in the isocapacitary estimate \eqref{iscap:omm'}.  }. By compactness, we may find a finite sequence of points $\{x^i\}_{i=1}^N\subset \partial \Omega$ such that
\begin{equation}\label{cover:deom}
   \partial \Omega\Subset \bigcup_{i=1}^N B_{\frac {R_\Omega}{8}}(x^i)\,,
\end{equation}
as well as $L_\Omega$-Lipschitz functions $\phi^i$ and isometries $T^i$ satisfying 
\begin{equation}\label{grafo1}
    \begin{split}
    & T^i \partial \Om \cap \big(B'_{R_\Omega}\times (-\ell,\ell)\big)=\big\{(y', \phi^i(y'))\,:\,y'\in B'_{R_\Omega}\big\},
    \\
    & T^i \Om \cap \big(B'_{R_\Omega}\times (-\ell,\ell)\big)=\big\{(y',y_n)\,:\,y'\in B'_{R_\Omega}\,,\,-\ell<y_n<\phi^i(y')\big\}.
\end{split}
\end{equation}
We denote by $\mathcal{R}^i$ the orthogonal matrix of $T^i$, i.e. $T^i$ can be written as
\begin{equation*}
    T^i x=\mathcal{R}^i(x-x^i)\quad x\in \R^n\,.
\end{equation*}

Notice also that the cardinality $N$ of this covering of $\partial \Omega$ may be chosen satisfying
\begin{equation}\label{card:NN}
    N\leq c(n)\,\bigg(\frac{d_\Omega}{R_\Omega}\bigg)^n\,.
\end{equation}
We then set
\begin{equation*}
    \Omega_t\coloneqq\{x\in \Omega\,:\, \mathrm{dist}(x,\partial\Omega)>t\}\,,
\end{equation*}
so that by \eqref{cover:deom} we have
\begin{equation}\label{cover:omega}
    \overline{\Omega}\Subset W\coloneqq \bigcup_{i=1}^N B_{\frac{R_\Omega}{8}}(x^i)\cup \Omega_{\frac{R_\Omega}{32}}\,.
\end{equation}
Starting from this point, we construct a suitable partition of unity: let
\begin{equation*}
    \eta_i\coloneqq \tilde{\rho}_{\frac{R_\Omega}{32}}\ast \chi_{B_{\frac{3R_\Omega}{16}}(x^i)}\quad\text{and}\quad \eta_0\coloneqq \tilde{\rho}_{\frac{R_\Omega}{64}}\ast \chi_{\Omega_{\frac{3R_\Omega}{64}}}\,,
\end{equation*}
where $\tilde{\rho}_{t}$ is the standard, radially symmetric convolution kernel on $\R^n$, and $\chi_A$ denotes the indicator function of a set $A$.

\noindent Standard properties of convolution ensure that $\eta_i\in C^\infty_c(B_{\frac{R_\Omega}{4}}(x^i))$, $\eta_0\in C^\infty_c(\Omega_{\frac{R_\Omega}{16}})$, \mbox{$0\leq \eta_i\leq 1$,} 
\begin{equation*}
    \eta_i\geq 1 \quad\text{on }B_{\frac{R_\Omega}{8}}(x^i)\,,\quad\eta_0\geq 1 \quad\text{on }\Omega_{\frac{R_\Omega}{32}}\,,
\end{equation*}
and
\begin{equation*}
    |\nabla^k \eta_i|\leq \frac{c(n,k)}{R_\Omega^k}\,,\quad\text{for all $k\in\N$.}
\end{equation*}
Therefore, by defining $\xi_i\,:\,W\to [0,1]$ as 
\begin{equation*}    
    \xi_i\coloneqq \frac{\eta_i}{\sqrt{\sum_{j=0}^N \eta_j}}\,,\quad i=0,\dots,N\,,
\end{equation*}
then we have that $\xi_i\in C^\infty_c(B_{\frac{R_\Omega}{4}}(x^i))$ for $i=1,\dots,N$, $\xi_0\in C^\infty_c(\Omega_{\frac{R_\Omega}{16}})$,
\begin{equation}\label{somma}
    \sum_{i=0}^N\xi_i(x)=1\quad\text{for all $x\in W$,}
\end{equation}
and
\begin{equation}\label{deriv:xi}
    |\nabla^k \xi_i|\leq \frac{c(n,k)}{R_\Omega^k}\quad\text{on $W$, for all $k\in\N$}\,.
\end{equation}

\subsection{Boundary defining function} Starting from the partition of unity $\{\xi_i\}_{i=0}^N$, and the local charts $\{\phi^i\}_{i=1}^N$, we can construct the boundary defining function of $\partial \Omega$ as in \cite[Proposition 5.43]{lee1}.

For any $\varepsilon\in [0,R_\Omega)$ and $j=1,\dots,N$, we define the rotated cylinders
\begin{equation}
    K^j_\varepsilon\coloneqq (T^j)^{-1}\big(B'_{R_\Omega-\varepsilon}\times (-\ell,\ell)\big)\,,
\end{equation}
where $\ell=R_\Omega(1+L_\Omega)$. Let $f^j:K^j_0\to \R$ be the functions defined as
\begin{equation*}
    f^j(x)\coloneqq z_n-\phi^j(z')\,,\quad z=T^j x\,,
\end{equation*}
and observe that from \eqref{grafo1} we have
\begin{equation}\label{edeome}
\begin{split}
    \{f^j=0\} & =\partial \Omega\cap K^j_0
     \\
     \{f^j<0\} & = \Omega\cap K^j_0
    \end{split}
\end{equation}

A boundary defining function of $ \overline{\Omega}$ is the function $F:W\to \R$ defined as
\begin{equation}\label{bd:defin}
    F(x)\coloneqq \sum_{j=1}^N f^j(x)\,\xi_j(x)-\xi_0(x)\,,
\end{equation}
where the product $f^j(x)\,\xi_j(x)$ is set equal to zero if $x\not\in\mathrm{supp}\,\xi_j$. Since each $f^j$ is Lipschitz continuous, so is the function $F$. 

Thanks to the properties of $\{\xi_j\}_{j=0}^N$, \eqref{grafo1} and \eqref{edeome}, it is easily seen that
\begin{equation}\label{bdef:om}
    \Omega=\{x\in W\,:\,F(x)<0\}\quad\text{and}\quad\partial\Omega=\{x\in W\,:\,F(x)=0\}\,.
\end{equation}

\subsection{Regularization and definition of the smooth approximating sets \texorpdfstring{$\omega_m,\Omega_m$}{}}
For $i=1,\dots,N$, we can define the smooth functions $\phi^i_m,\widetilde{\phi}^i_m:B'_{R_\Omega-\frac{1}{m}}\to \R$ as
\begin{gather}
    \phi^i_m\coloneqq M_m(\phi^i)+\|M_m(\phi^i)-\phi^i\|_{L^\infty(B'_{R_\Omega-1/m})}+\frac{L_\Omega}{m}   \notag
    \\
    \text{and} \label{def:regular}
    \\
    \widetilde{\phi}^i_m\coloneqq M_m(\phi^i)-\|M_m(\phi^i)-\phi^i\|_{L^\infty(B'_{R_\Omega-1/m})}-\frac{L_\Omega}{m}\,.\notag
\end{gather}
From the results of Proposition \ref{pr:rotconv}, we deduce that $\phi^i_m,\widetilde{\phi}^i_m\in C^\infty$ are $L_\Omega$-Lipschitz functions, and 
\begin{equation}\label{linf:m}
    \begin{split}
        \frac{L_\Omega}{m}\leq \phi^i_m(y')-\phi^i(y')\leq\frac{3\,L_\Omega}{m}
        \\
        \frac{L_\Omega}{m}\leq \phi^i(y')-\widetilde{\phi}^i(y')\leq\frac{3\,L_\Omega}{m}\,,
    \end{split}
\end{equation}
for all $y'\in B'_{R_\Omega-1/m}$ and $i=1,\dots,N$.
Taking inspiration from \eqref{edeome} and \eqref{bdef:om}, we are led to define the functions
\begin{equation}\label{effe:jm}
\begin{split}
    & f^j_m(x)\coloneqq z_n-\phi^j_m(z')
    \\
    & \tilde{f}^j_m(x)\coloneqq z_n-\widetilde{\phi}^j_m(z')\,,\quad z=T^j x\in B'_{R_\Omega-\frac{1}{m}}\times(-\ell,\ell)\,,
\end{split}
\end{equation}
and functions $F_m,\widetilde{F}_m:W\to \R$ defined as
\begin{equation}\label{bdm:defin}
    \begin{split}F_m(x)\coloneqq\sum_{j=1}^N f^j_m(x)\,\xi_j(x)-\xi_0(x)
    \\
    \widetilde{F}_m(x)\coloneqq\sum_{j=1}^N \tilde{f}^j_m(x)\,\xi_j(x)-\xi_0(x)\,,
\end{split}
\end{equation}
where the products $f^j_m(x)\,\xi_j(x)$ and $\tilde{f}^j_m(x)\,\xi_j(x)$ have to be interpreted equal to zero when $x\not\in \mathrm{supp}\,\xi_j$.

Clearly, $F_m$ and $\widetilde{F}_m$ are $C^\infty$-smooth functions on $W$, and since 
\begin{equation}\label{efjmef}
    \frac{L_\Omega}{m}\leq f^j(x)-f^j_m(x)<\frac{3\,L_\Omega}{m}\,,\quad \frac{L_\Omega}{m}\leq\tilde{f}^j_m(x)-f^j(x)<\frac{3\,L_\Omega}{m}
\end{equation}
for all $x\in K^j_{1/m}$ thanks to \eqref{linf:m}, we then have
\begin{equation}\label{iop}
    \frac{L_\Omega}{m}\leq F(x)-F_m(x)\leq \frac{3\,L_\Omega}{m}\,,\quad \frac{L_\Omega}{m}\leq\widetilde{F}_m(x)-F(x)\leq\frac{3\,L_\Omega}{m}\quad\text{for all $x\in W$.}
\end{equation}

The approximating open sets $\Omega_m,\,\omega_m$ are thus defined as follows
\begin{equation}
    \Omega_m\coloneqq \{x\in W\,:\,F_m(x)<0\}\quad\text{and}\quad \omega_m\coloneqq \{x\in W\,:\,\widetilde{F}_m(x)<0\}\,,
\end{equation}
with boundaries
\begin{equation}
    \partial\Omega_m= \{x\in W\,:\,F_m(x)=0\}\quad\text{and}\quad \partial\omega_m= \{x\in W\,:\,\widetilde{F}_m(x)=0\}\,.
\end{equation}
In particular, since $F_m(x)<F(x)<\widetilde{F}_m(x)$ for all $x\in W$, owing to \eqref{bdef:om} we have
\begin{equation*}
    \omega_m\Subset \Omega\Subset \Omega_m\quad\text{for all $m\in \N$.}
\end{equation*}

We now proceed to prove the remaining properties of Theorem \ref{appr:mthm} for the outer sets $\Omega_m$.  The proofs for the inner sets $\omega_m$ are analogous.

\subsection{\texorpdfstring{$\partial \Omega_m,\partial \omega_m$}{} are smooth manifolds .}

Let us show that $\partial \Omega_m$ is a smooth manifold, with local charts $\{\psi^i_m\}_{i=1}^N$ defined on the same coordinate systems as $\{\phi^i\}_{i=1}^N$.

We fix a constant $\varepsilon_0\in (0,R_\Omega/4)$, and for all $i=1,\dots,N$ we set
\begin{equation*}
    F^i(y)=F(x) 
 \quad\text{and}\quad F^i_m(y)=F_m(x)\quad\text{for $y=T^i x$, $x\in W$.}
\end{equation*}

Owing to \eqref{grafo1} we have 
\begin{equation}\label{prep:transition0}
    \begin{split}
        \partial \Omega\cap K^i_0\cap K^j_0=(T^i)^{-1} & G_{\phi^i}\cap K^j_0=(T^j)^{-1}G_{\phi^j}\cap K^i_0  
        \\
        & \text{and}
        \\
         \Omega\cap K^j_0\cap K^i_0=(T^i)^{-1} & S_{\phi^i}\cap K^j_0\cap  K^i_0=(T^j)^{-1}S_{\phi^j}\cap K^i_0\cap  K^j_0\,,
    \end{split}
\end{equation}
whenever $ \partial \Omega\cap K^i_0\cap K^j_0\neq \emptyset$.

\noindent This piece of information will allow us to use the transversality property. Specifically, thanks to \eqref{prep:transition0} we may apply  Propositions \ref{prop:trass}-\ref{pr:rotconv} with functions $\phi=\phi^j$, $\psi=\phi^i$, isometry $T=T^j(T^i)^{-1}$, and defining set
\begin{equation*}
    U=U^{j,i}=\Pi\Big( G_{\phi^j}\cap T^j K^i_0\Big)\subset B'_{R_\Omega}\,.
\end{equation*}

\noindent\textbf{Claim 1.} There exists $m_0>0$ such that, for all $i=1,\dots,N$, for all  $m\geq m_0$ and all $x\in \big\{\frac{-3L_\Omega}{m_0}\leq F\leq \frac{3L_\Omega}{m_0}\big\}\cap K^i_{\varepsilon_0}$, we have
\begin{equation}\label{coervic:Fm}
    \frac{\partial F^i_m}{\partial y_n}(y)\geq \frac{1}{2\sqrt{1+L_\Omega^2}}\,,\quad\text{for all  $y=T^i x\in B'_{R_\Omega-\varepsilon_0}\times (-\ell,\ell)$. }
\end{equation}
Suppose by contradiction this is false; then for every $k\in \N$, we may find $m_k\geq k$ and a sequence $x^k\in \big\{-\frac{3L_\Omega}{k}\leq F\leq \frac{3L_\Omega}{k}\big\}$ such that $y^k=T^i x^k\in B'_{R_\Omega-\varepsilon_0}\times (-\ell,\ell)$ and
\begin{equation}\label{contradiction:1}
    \frac{\partial F_{m_k}^i}{\partial y_n}(y^k)<\frac{1}{2\sqrt{1+L_\Omega^2}}\,,\quad \text{for all $k\in \N$}
\end{equation}
By compactness, we may extract a subsequence, still labeled as $x^k$, such that $x^k\to x^0$, and in particular  $x^0\in \overline{K^i_0}$ and $F(x^0)=0$, hence $x^0\in \partial \Omega\cap \overline{K^i_0}$ due to \eqref{bdef:om}.

Then, by the chain rule we have
\begin{equation}\label{hbrfh}
    \frac{\partial f^i_m}{\partial y_n}(x)=1\quad\text{and}\quad \frac{\partial f^j_m}{\partial y_n}(x)=\big(\mathcal{R}^j(\mathcal{R}^i)^t\big)_{nn}-\sum_{s=1}^{n-1}\frac{\partial \phi^{j}_m}{\partial z'_s}\big(z' \big)\,\big(\mathcal{R}^j(\mathcal{R}^i)^t\big)_{ns}\,,
\end{equation}
if $x\in \mathrm{supp}\,\xi_j$, where $z'=\Pi\,T^j x$.
We now distinguish two cases: \\(i) $j\in \{1,\dots,N\}$ is such that $x^0\not\in \mathrm{supp}\,\xi_j$. Then $\mathrm{dist}\big(x^0,\mathrm{supp}\,\xi_j\big)>0$, hence $x^k\not\in \mathrm{supp}\,\xi_j$ for all $k\geq k_0$ large enough.
\\ (ii) $j\in \{1,\dots,N\}$ is such that $x^0\in \mathrm{supp}\,\xi_j$. In this case, it follows that $x^0\in \partial \Omega\cap K^i_0\cap B_{\frac{R_\Omega}{4}}(x^j)$, so that from \eqref{prep:transition0} we have $T^j x^0\in G_{\phi^j}\cap B_{\frac{R_\Omega}{4}}\cap T^j \overline{K^i_{\varepsilon_0}}$. By setting $(z^k)'=\Pi\,T^j x^k$, we thus have
\begin{equation*}
B'_{\frac{1}{m_k}}\big((z^k)'\big)\Subset\Pi\Big( G_{\phi^j}\cap T^j K^i_0\Big) \,,
\end{equation*}
for all $k\geq k_0$ large enough.
 Recalling the remarks after \eqref{prep:transition0},
by applying Proposition \ref{pr:rotconv}, and in particular  
 the transversality property \eqref{quasi:transM}
in \eqref{hbrfh}, we infer
\begin{equation*}
    \frac{\partial f^j_{m_k}}{\partial y_n}\big(x^k\big)=\big(\mathcal{R}^j(\mathcal{R}^i)^t\big)_{nn}-\sum_{s=1}^{n-1}\frac{\partial \phi^{j}_{m_k}}{\partial z'_s}\big((z^k)' \big)\,\big(\mathcal{R}^j(\mathcal{R}^i)^t\big)_{ns}\geq \frac{1}{\sqrt{1+L^2_\Omega}}\,,
\end{equation*}
provided $k\geq k_0$ is large enough. 
\\
In both cases, we have found that
\begin{equation}\label{tempo:cerFm}
    \frac{\partial f^j_{m_k}}{\partial y_n}(x^k)\,\xi_j\big(x^k\big)\geq\frac{\xi_j(x^k)}{\sqrt{1+L_\Omega^2}}\quad\text{for all $j=1,\dots,N$ and $k\geq k_0$.}
\end{equation}
Also, owing to \eqref{efjmef} and \eqref{edeome} we have
\begin{equation*}
    \begin{split}
        |f^j_{m_k}(x^k)|\,\Big|\frac{\partial \xi_j(x^k)}{\partial y_n}\Big| & \leq |f^j_{m_k}(x^k)-f^j(x^k)|\,|\nabla\xi_j(x^k)|+|f^j(x^k)|\,|\nabla\xi_j(x^k)|
        \\
        &\leq \frac{1}{m_k}+|f^j(x^k)|\,|\nabla\xi_j(x^k)|\xrightarrow{k\to\infty} |f^j(x^0)|\,|\nabla\xi_j(x^0)|=0\,,
    \end{split}
\end{equation*}
and $|\nabla \xi_0(x^k)|\to |\nabla\xi_0(x^0)|=0$ since $x^0\in \partial \Omega$. By coupling this piece of information with \eqref{somma}, \eqref{contradiction:1} and \eqref{tempo:cerFm}, we finally obtain 
\begin{equation*}
\begin{split}
    \frac{1}{2\sqrt{1+L_\Omega^2}}  >\frac{\partial F^i_{m_k}}{\partial y_n}(y^k) & =\sum_{j=1}^N\frac{\partial f^j_{m_k}}{\partial y_n}(x^k)\,\xi_j(x^k)+\sum_{j=1}^N f^j_{m_k}(x^k)\,\frac{\partial \xi_j}{\partial y_n}(x^k)-\frac{\partial \xi_0}{\partial y_n}(x^k)
    \\
    & \geq \sum_{j=1}^N \frac{\xi_j(x^k)}{\sqrt{1+L_\Omega^2}}+\sum_{j=1}^N f^j_{m_k}(x^k)\,\frac{\partial \xi_j}{\partial y_n}(x^k)-\frac{\partial \xi_0}{\partial y_n}(x^k)
    \\
    & \xrightarrow{k\to\infty}\sum_{j=1}^N\frac{\xi_j(x^0)}{\sqrt{1+L_\Omega^2}}=\frac{1}{\sqrt{1+L_\Omega^2}}\,,
    \end{split}
\end{equation*}
which is a contradiction, and thus \eqref{coervic:Fm} holds true.
\vspace{0.3cm}

\noindent\textbf{Claim 2.} There exists $m_0>0$ such that $\forall y'\in B'_{R_\Omega-\varepsilon_0}$, $\forall m\geq m_0$, $\exists y_n\in (-\ell,\ell)$ with $y=(y',y_n)=T^i x\in T^i W$ satisfying $F_m^i(y)\geq 0$.

\vspace{0.2cm}

Again, assume by contradiction this is false. Then for all $k\in \N$, we may find sequences $m_k\geq k$ and $(y^k)'\in B'_{R_\Omega-\varepsilon_0}$ such that 
\begin{equation}\label{contradiction:2}
    F^i_{m_k}\big( (y^k)',y_n)<0\quad\text{for all $y_n\in (-\ell,\ell)$ such that $\big((y^k)',y_n\big)\in T^i W$.}
\end{equation}
By compactness, we may find a subsequence, still labeled as $(y^k)'$, satisfying $(y^k)'\to (y^0)'\in \overline{B}'_{R_\Omega-\varepsilon_0}$. Fix $w_n\in (-\ell,\ell)$ such that $\big((y^0)',w_n\big)\in T^i W$, and let $\{w^k_n\}_{k\in \N}\subset \R$ be a sequence satisfying $w^k_n\xrightarrow{k\to\infty} w_n$. Then $\big( (y^k)', w_n^k\big)\to \big( (y^0)', w_n\big)$, so that $\big( (y^k)', w_n^k)\in T^i W$ for $k\geq k_0$ large enough being $W$ open, and from \eqref{contradiction:2} we have $F^i_{m_k}\big((y^k)',w_n^k \big)<0$. By using \eqref{iop} and the Lipschitz continuity of $F$, it is readily shown that
\begin{equation*}
    \lim_{k\to\infty} F^i_{m_k}\big( (y^k)',w_n^k)=F^i\big((y^0)', w_n)\,,
\end{equation*}
whence $F^i\big((y^0)', w_n)\leq 0$ for all $w_n$ as above, but this contradicts the fact that \mbox{$F^i\big( (y^0)',w_n\big)>0$}
 whenever $w_n>\phi^i\big((y^0)'\big)$ due to \eqref{bdef:om}, hence Claim 2 is proven.
\vspace{0.3cm}

Now let $y'\in B'_{R_\Omega-\varepsilon_0}$; by \eqref{iop} and since $F^i\big(y',\phi^i(y')\big)=0$, we have \mbox{$F^i_m\big(y',\phi^i(y')\big)<0$}. Thus, owing to Claim 2 we may find $y_n$ such that $F^i_m(y',y_n)=0$.

The monotonicity property \eqref{coervic:Fm} of Claim 1, and the fact that $\partial \Omega_m=\{F_m=0\}\subset \{\frac{L_\Omega}{m}\leq F\leq \frac{3L_\Omega}{m}\}$ due to \eqref{iop} ensure that such point $y_n$ is unique for all $y'\in B'_{R_\Omega-\varepsilon_0}$.
This entails the existence of a function $\psi^i_m\,:\,B'_{R_\Omega-\varepsilon_0}\to \R$ such that $F^i_m\big(y',\psi^i_m(y')\big)=0$ for all $y'\in B'_{R_\Omega-\varepsilon_0}$. Furthemore, owing to \eqref{bdef:om} and \eqref{iop}, we have that $\psi^i_m(y')>\phi^i(y')$ for all $y'\in B'_{R_\Omega-\varepsilon_0}$, and from the implicit function theorem we also infer that $\psi^i_m\in C^\infty\big(B'_{R_\Omega-\varepsilon_0}\big)$.
Moreover, via a compactness argument as in Claim 1-2 and \eqref{cover:deom}, one can prove that
\begin{equation}\label{conenuto}
\begin{split}
   \bigg\{-\frac{3L_\Omega}{m}\leq\, &\, F\leq \frac{3 L_\Omega}{m}\bigg\}\subset \bigcup_{i=1}^N B_{\frac{R}{8}}(x^i)
    \\
  \bigg\{-\frac{3L_\Omega}{m}\leq\, &\, F\leq \frac{3 L_\Omega}{m}\bigg\}\cap \mathrm{supp}\,\xi_0=\emptyset\,,\quad\text{for all $m>m_0$,}
    \end{split}
\end{equation}
so that, in particular, the cylinders $\big\{K^i_{2\varepsilon_0}\big\}_{i=1}^N$ are an open cover of $\partial \Omega_m$, and \mbox{$\partial \Omega_m\cap\, \mathrm{supp}\,\xi_0=\emptyset$} provided  $m>m_0$ is large enough.

We have thus proven that $\partial \Omega_m$ is a  $C^\infty$-smooth manifold for $m>m_0$, with local boundary charts $\{\psi^i_m\}_{i=1}^N$ defined on the same coordinate cylinders as $\{\phi^i\}_{i=1}^N$, that is 
\begin{equation}\label{grafo2}
    \begin{split}
    & T^i \partial \Om_m \cap \big(B'_{R_\Omega-\varepsilon_0}\times (-\ell,\ell)\big)=\big\{(y', \psi^i_m(y'))\,:\,y'\in B'_{R_\Omega-\varepsilon_0}\big\},
    \\
    & T^i \Om_m \cap \big(B'_{R_\Omega-\varepsilon_0}\times (-\ell,\ell)\big)=\big\{(y',y_n)\,:\,y'\in B'_{R_\Omega-\varepsilon_0}\,,\,-\ell<y_n<\psi^i_m(y')\big\}.
\end{split}
\end{equation}

\subsection{Approximation properties.}
First, we show that there exists $m_0>0$ such that
\begin{equation}\label{uniffff}
    \|\psi^i_m-\phi^i\|_{L^\infty(B'_{R_\Omega-2\,\varepsilon_0})}\leq \frac{6\,L_\Omega\,\sqrt{1+L_\Omega^2}}{m}\quad\text{for all $m>m_0$.}
\end{equation}
Assume by contradiction this is false; then we may find  sequences $m_k\uparrow\infty$ and $(y^k)'\in B'_{R_\Omega-2\varepsilon_0}$ such that
\begin{equation}\label{contradiction:3}
    \psi^i_{m_k}\big((y^k)' \big)-\phi^i\big((y^k)'\big)> \frac{6\,L_\Omega\,\sqrt{1+L_\Omega^2}}{m_k}
\end{equation}
Up to a subsequence, we have $(y^k)'\to (y^0)'\in \overline{B}'_{R_\Omega-2\varepsilon_0}$, and $\psi^i_{m_k}\big((y^k)'\big)\to \ell_0\in \R$. Furthermore, since $\Big( (y^k)',\psi^i_m\big((y^k)'\big)\Big)\in \{F^i_{m_k}=0\}\subset T^i\{\frac{L_\Omega}{m_k} \leq F\leq \frac{3L_\Omega}{m_k}\}$, we readily infer that $F^i\big((y^0)',\ell_0 \big)=0$, whence $\ell_0=\phi^i\big( (y')^0\big)$ due to \eqref{bdef:om} and \eqref{grafo1}. By continuity we also have $\phi^i\big((y^k)'\big)\to \phi^i\big((y^0)'\big)$, which implies that
\begin{equation*}
    \psi^i_{m_k}\big((y^k)' \big)-\phi^i\big((y^k)'\big)\xrightarrow{k\to\infty}0\,.
\end{equation*}
Then, for all $t\in[0,1]$, we have
\begin{equation*}
    \begin{split}
        \Big|F^i\Big((y^k)', t\,\psi^i_{m_k}\big((y^k)' \big) +(1-t) &\phi^i\big((y^k)'\big) \Big)-F^i\Big( (y^k)',\phi^i\big((y^k)' \big)\Big)\Big|
        \\
        & \leq L_F\,t\,|\psi^i_{m_k}\big((y^k)' \big)-\phi^i\big((y^k)'\big)|\xrightarrow{k\to\infty}0\,,
    \end{split}
\end{equation*}
where $L_F$ denotes the Lipschitz constant of $F$. This implies that for all $k\geq k_0$ large enough, the line segment
\begin{equation*}
    \big\{(y^k)'\big\}\times \big[\phi^i\big((y^k)' \big),\psi^i_{m_k}\big((y^k)' \big) \big]\subset T^i\Big\{ -\frac{3\,L_\Omega}{m_0}\leq F\leq \frac{3\,L_\Omega}{m_0}\Big\}\,. 
\end{equation*}
Therefore, by using \eqref{grafo1}, \eqref{bdef:om} \eqref{iop}, \eqref{coervic:Fm} and \eqref{contradiction:3}, we obtain
\begin{equation*}
    \begin{split}
        \frac{3\,L_\Omega}{m_k} & \geq F^i\Big((y^k)',\phi^i\big((y^k)' \big) \Big)-F^i_{m_k}\Big((y^k)',\phi^i\big((y^k)' \big) \Big)=-F^i_{m_k}\Big((y^k)',\phi^i\big((y^k)' \big) \Big)
        \\
        & =F^i_{m_k}\Big((y^k)',\psi^i_{m_k}\big((y^k)' \big)\Big)-F^i_{m_k}\Big((y')^k,\phi^i\big((y')^k \big) \Big)
        \\
        & =\Bigg(\int_0^1\frac{\partial F^i_{m_k}}{\partial y_n}\Big((y^k)',t\,\psi^i_{m_k}\big((y^k)'\big)+(1-t)\,\phi^i\big( (y^k)'\big) \Big)\,dt \Bigg)\,\Big[\psi^i_{m_k}\big((y^k)'\big)-\phi^i \big((y^k)'\big) \Big]
        \\
        &> \frac{1}{2\sqrt{1+L_\Omega^2}}\,\frac{6\,L_\Omega\,\sqrt{1+L^2}}{m_k}=\frac{3\,L_\Omega}{m_k}\,,\quad\text{for all $k\geq k_0$ large enough,}
    \end{split}
\end{equation*}
which is a contradiction, hence \eqref{uniffff} holds true.
\medskip

Now, recalling that $\{K^j_{2\varepsilon_0}\}_{j=1}^N$ is an open cover of $\partial \Omega$ and $\partial \Omega_m$, from \eqref{grafo1}, \eqref{grafo2} and \eqref{uniffff}, one can easily obtain that 
\begin{equation*}
    \mathrm{dist}_\H\big(\partial \Omega_m,\partial \Omega)\leq\frac{6\,L_\Omega\sqrt{1+L_\Omega^2}}{m}\,.
\end{equation*}
This convergence property in the sense of Hausdorff immediately implies that $d_{\Omega_m}\leq c(n)\,d_\Omega$, and $\lim_{m\to\infty}|\Omega_m\setminus\Omega|=0$ --see for instance \cite[Proposition 2.2.23]{henrot}-- and thus \eqref{diameters}, \eqref{leb:distance} and \eqref{hauss:dist} are proven.
\medskip

Let us now introduce the transition maps related to the local charts of $\partial \Omega$ and $\partial \Omega_m$.

First of all, note that thanks to \eqref{grafo2}, we have 
\begin{equation}\label{prep:transition}
    \begin{split}
        \partial \Omega_m\cap K^i_{\varepsilon_0}\cap K^j_{\varepsilon_0}=(T^i)^{-1} & G_{\psi^i_m}\cap K^j_{\varepsilon_0}=(T^j)^{-1}G_{\psi^j_m}\cap K^i_{\varepsilon_0} 
        \\
        & \text{and}
        \\
         \Omega_m\cap K^j_{\varepsilon_0} \cap K^i_{\varepsilon_0} =(T^i)^{-1} & S_{\psi^i_m}\cap K^j_{\varepsilon_0}\cap K^i_{\varepsilon_0}  =(T^j)^{-1}S_{\psi^j_m}\cap K^i_{\varepsilon_0} \cap K^j_{\varepsilon_0}\,,  
    \end{split}
\end{equation}
whenever $\partial \Omega_m\cap K^i_{\varepsilon_0}\cap K^j_{\varepsilon_0}\neq \emptyset$.

For all $i\in\{1,\dots,N\}$, we define the set of indexes
\begin{equation*}
    \mathcal{I}_i\coloneqq\big\{ j\in\{1,\dots,N\}\,:\,\partial \Omega\cap K^i_{2\varepsilon_0}\cap K^j_{2\varepsilon_0}\neq \emptyset\big\}\,.
\end{equation*}
If $j\in \mathcal{I}_i$, then owing to \eqref{grafo1} there exists $y'\in B'_{R_\Omega-2\varepsilon_0}$ such that $(T^i)^{-1}\big(y',\phi^i(y')\big)\in \partial \Omega\cap K^j_{2\varepsilon_0}$. Since $\phi^j$ is $L_\Omega$-Lipschitz continuous and $\phi^j(0')=0$, we have $|\phi^j(z')|\leq L_\Omega\,|z'|$, so it follows from \eqref{prep:transition0}, \eqref{grafo2} and \eqref{uniffff} that 
$(T^i)^{-1}\big(y',\psi^i_m(y')\big)\in \partial \Omega_m\cap K^i_{\varepsilon_0}\cap K^j_{\varepsilon_0}$ for all $m\geq m_0$ large enough.

Henceforth, for all $j\in \mathcal{I}_i$,  \eqref{prep:transition0}  and \eqref{prep:transition} allow us to define the transition maps $\C^{i,j}, \C^{i,j}_m$  from $\phi^i$ to $\phi^j$ and from $\psi^i_m$ to $\psi^j_m$ respectively, i.e.
\begin{equation}
    \begin{split}
    \C^{i,j}y' & =\Pi\,T^j(T^i)^{-1}\big(y',\phi^i(y')\big)
    \\
    \C^{i,j}_m y' & =\Pi\,T^j(T^i)^{-1}\big(y',\psi^i_m(y')\big)\,,
    \end{split}
\end{equation}
which are defined on the open sets
\begin{equation*}
    \begin{split}
    U^{i,j}  =\Pi\,\Big(G_{\phi^i}\cap T^i\,K^j_0\Big)\quad\text{and}\quad
    U^{i,j}_m  =\Pi\,\Big(G_{\psi^i_m}\cap T^i\,K^j_{\varepsilon_0}\Big)\,.
    \end{split}
\end{equation*}
In particular, by their definitions and the arguments of Section \ref{sec:tras}, we may write
\begin{equation}\label{ij:transition}
\begin{split}
     x=(T^i)^{-1}\big(y',\phi^i(y')\big)=(T^j)^{-1}\big(\C^{i,j}y',\phi^j(\C^{i,j}y')\big)\quad & \text{for $x\in \partial \Omega\cap K^i_0\cap K^j_0$}
    \\
     x^m=(T^i)^{-1}\big(y',\psi^i_m(y')\big)=(T^j)^{-1}\big(\C_m^{i,j}y',\psi_m^j(\C^{i,j}_my')\big)\quad & \text{for $x^m\in \partial \Omega_m\cap K^i_{\varepsilon_0}\cap K^j_{\varepsilon_0}$.}
\end{split}
\end{equation}
and their inverse functions are $(\C^{i,j})^{-1}=\C^{j,i}$ and $(\C^{i,j}_m)^{-1}=\C^{j,i}_m$.
Observe also that $\C^{i,i}=\C^{i,i}_m=\mathrm{Id}$.

Furthermore, since $\mathrm{supp}\,\xi_j\Subset B_{R_\Omega/4}(x^j)\Subset K^j_{2\varepsilon_0}$, it follows from the definition of $\mathcal{I}_i$ and \eqref{uniffff} that 
\begin{equation}\label{non:Ii}
    \xi_j\big( (T^i)^{-1}(y',\phi^i(y'))\big)=\xi_j\big( (T^i)^{-1}(y',\psi^i_m(y'))\big)=0\quad\text{if $j\not\in \mathcal{I}_i$,}
\end{equation}
for all  $y'\in B'_{R_\Omega-\varepsilon_0}$, and all $m\geq m_0$.

We now claim that for all $j\in\mathcal{I}_i$, there exists an open set $V^{i,j}\subset B'_{R_\Omega-2\varepsilon_0}$ for which we have 
\begin{equation}\label{nulla:partiz}
    \xi_j\big( (T^i)^{-1}(y',\phi^i(y'))\big)=\xi_j\big( (T^i)^{-1}(y',\psi^i_m(y'))\big)=0\quad\text{if $y'\not\in V^{i,j}$,}
\end{equation}
and such that $V^{i,j}\subset U^{i,j}\cap U^{i,j}_m$ for all $m>m_0$. This in particular implies that both $\C^{i,j}$ and $\C^{i,j}_m$ are defined on $V^{i,j}$.

To this end, let
\begin{equation*}
    V^{i,j}\coloneqq \Pi\Big( G_{\phi^i}\cap T^i K^j_{2\varepsilon_0}\Big)\cap B'_{R_\Omega-2\varepsilon_0}\,.
\end{equation*}
Then, owing to \eqref{uniffff} it is immediate to verify that
\begin{equation}\label{mam:mia}
   B'_{R_\Omega-2\varepsilon_0}\cap\bigg( \Pi\Big( G_{\phi^i}\cap T^i B_{R_\Omega/4}(x^j)\Big)\cup \Pi\Big( G_{\psi_m^i}\cap T^i B_{R_\Omega/4}(x^j)\Big)\bigg)\Subset V^{i,j}\,,
\end{equation}
whenever $m>m_0$ is large enough, and thus \eqref{nulla:partiz} is satisfied by our choice of set $V^{i,j}$.

Clearly  $V^{i,j}\subset U^{i,j}$, so we are left to verify that $V^{i,j}\subset U^{i,j}_m$.
To this end, let $y'\in V^{i,j}$; then by \eqref{prep:transition} and \eqref{ij:transition} we may write
\begin{equation*}
    T^j(T^i)^{-1}(y',\phi^i(y'))=(\C^{i,j}y',\phi^j(\C^{i,j}y'))\in B'_{R_\Omega-2\varepsilon_0}\times \big(-L_\Omega(R_\Omega-2\varepsilon_0),L_\Omega(R_\Omega-2\varepsilon_0)\big)\,,
\end{equation*}
 where in the latter inclusion we made use of the inequality $|\phi^j(z')|\leq L_\Omega\,|z'|$.
 Therefore, thanks to \eqref{uniffff}, for $m>m_0$ we  have $(T^i)^{-1}(y',\psi^i_m(y'))\in\partial \Omega_m\cap  K^i_{\varepsilon_0}\cap K^j_{2\varepsilon_0}$, hence $y'\in U^{i,j}_m$ by \eqref{prep:transition} and the definition of $U^{i,j}_m$, so the claim is proven.

We also remark that
\begin{equation}\label{tuttovij}
    \bigcup_{j\in\mathcal{I}_i}V^{i,j}=B'_{R_\Omega-2\varepsilon_0}\,,
\end{equation}
since $\{T^i K^j_{2\varepsilon_0}\}_{j\in \mathcal{I}_i}$ is an open cover of $G_{\phi^i}\cap K^i_{2\varepsilon_0}$, and the projection map $\Pi$ is a homeomorphism from $G_{\phi^i}$ (with the induced topology) to $B'_{R_\Omega}$.

Moreover, owing to \eqref{uniffff} and by proceeding as in the derivation of \eqref{uoaaa1}, we obtain
\begin{equation}\label{uoaaa2}
    \|\C^{i,j}_m-\C^{i,j}\|_{L^\infty(V^{i,j})}\leq \frac{6\,L_\Omega\sqrt{1+L_\Omega^2}}{m}\quad\text{for all $m>m_0$.}
\end{equation}




Our next goal is to obtain estimates on $\nabla \psi^i_m$. To this end, we differentiate equation $F_m^i(y',\psi^i_m(y'))=0$ with respect to $y'_k$, for $k=1,\dots,n-1$, and recalling \eqref{non:Ii}  we
 find
\begin{equation}\label{form:derpsi}
\frac{\partial\psi^i_m }{\partial y'_k}(y')=-\bigg(\frac{\partial F_m^i(y',\psi^i_m(y'))}{\partial y_n}\bigg)^{-1}\sum_{j\in \mathcal{I}_i}\Bigg\{ \frac{\partial f^j_m(x^m)}{\partial y_k'}\,\xi_j(x^m)+f_m^j(x^m)\,\frac{\partial \xi_j(x^m)}{\partial y_k'}\Bigg\}\,,
\end{equation}
where $x^m=(T^i)^{-1}\big(y',\psi^i_m(y')\big)$, $y'\in B'_{R_\Omega-2\varepsilon_0}$. 

For all $l=1,\dots,n$, by using the chain rule and recalling the definition of $\C^{i,j}_m$, we find
\begin{equation}\label{fm:der}
\begin{split}
     & \frac{\partial f^i_m}{\partial y'_l}(x^m)=-\frac{\partial \phi_m^i}{\partial y'_l}(y')\quad \text{and}\quad\frac{\partial f^i_m}{\partial y_n}(x^m)=1
     \\
      & \frac{\partial f^j_m}{\partial y_l}(x^m)=\mathcal{R}^i(\mathcal{R}^j)^t_{nl}-\sum_{r=1}^{n-1}\frac{\partial \phi^j_m}{\partial z'_r}(\C^{i,j}_m y')\,\mathcal{R}^i(\mathcal{R}^j)^t_{rl}\,,
     \end{split}
\end{equation}
for all $j\in \mathcal{I}_i$ such that $x^m\in \mathrm{supp}\,\xi_j$. Since $\phi^j_m$ are $L_\Omega$-Lipschitz continuous, from \eqref{fm:der} it follows that
\begin{equation}\label{nfjxm}
   \sum_{l=1}^n \Big|\frac{ \partial f^j_m(x^m)}{\partial y_l}\Big|\leq c(n)(1+L_\Omega),\quad\text{for all $j\in\mathcal{I}_i$.}
\end{equation}
 Moreover, from \eqref{efjmef}, 
 \eqref{uniffff} and \eqref{edeome}, we find that \mbox{$f^j_m(x^m)\,|\nabla\xi_j(x^m)|\xrightarrow{m\to \infty} f^j(x^0)\,|\nabla\xi_j(x^0)|=0$}, where $x^0=(T^i)^{-1}\big(y',\phi^i(y')\big)\in \partial \Omega$.

By making use of this piece of information, \eqref{nfjxm} and \eqref{coervic:Fm}, from \eqref{form:derpsi} we finally obtain the gradient estimate
\begin{equation}\label{lipschitz:m}
    |\nabla \psi^i_m(y')|\leq c(n)\big(1+L_\Omega^2\big)\,,\quad\text{for all $y'\in B'_{R_\Omega-2\varepsilon_0}$,}
\end{equation}
 for all $i=1,\dots,N$ and $m>m_0$ large enough. In particular, owing to \eqref{uniffff}, \eqref{grafo2} and \eqref{lipschitz:m}, it is readily seen that $\Omega_m$ are $\mathcal{L}_{\Omega_m}$-Lipschitz domains, with
 \begin{equation*}
     L_{\Omega_m}\leq c(n)\big(1+L_\Omega^2\big)\quad\text{and}\quad R_{\Omega_m}\geq \frac{R_\Omega}{c(n)\,\big(1+L_\Omega^2\big)}\,,
 \end{equation*}
 and \eqref{lip:car} is proven.
\medskip

Next, the definition of $\C^{i,j}$ and $\C^{i,j}_m$, \eqref{lipschitz:m} and the $L_\Omega$-Lipschitz continuity of $\phi^i$ imply
\begin{equation}\label{lip:transmap}
 \sup_{i=1,\dots,N}\sup_{j\in\mathcal{I}_i}\Big\{\|\nabla \C^{i,j}\|_{L^\infty}+\|\nabla \C^{i,j}_m\|_{L^\infty}\Big\}\leq c(n)(1+L_\Omega^2)\quad\text{for all $m>m_0$,} 
\end{equation}
and in particular $\C^{i,j}$ and $\C^{i,j}_m$ are uniformly bi-Lipschitz transformations.

Hence, thanks to \eqref{uoaaa2} and \eqref{lip:transmap}, we are in the position to apply Proposition \ref{prop:convoluz} and get
\begin{equation}\label{puntual}
    \frac{\partial \phi^j_m}{\partial z'_r}(\C^{i,j}_m y')\xrightarrow{m\to\infty} \frac{\partial \phi^j}{\partial z'_r}(\C^{i,j}y')\quad\text{for $\H^{n-1}$-a.e. $y'\in V^{i,j}$.}
\end{equation}
From this, \eqref{coervic:Fm}, \eqref{nulla:partiz},  \eqref{tuttovij},  \eqref{fm:der} and identity \eqref{form:derpsi}
 we find
 \begin{equation*}
     \nabla\psi_m^i(y')\xrightarrow{m\to\infty} G(y')\quad\text{for $\H^{n-1}$-a.e. $y'\in B'_{R_\Omega-2\varepsilon_0}$}\,,
 \end{equation*}
where $G$ is a bounded vector valued function which can be explictly written. From \eqref{lipschitz:m} and on applying dominated convergence theorem, we get that $\nabla \psi^i_m\xrightarrow{m\to\infty} G$ in $L^p(B'_{R-2\varepsilon_0})$ for all $p\in [1,\infty)$. On the other hand, \eqref{uniffff} and the uniqueness of the distributional limit imply that $G=\nabla \phi^i$, hence \eqref{lchart:conv1} is proven. 

\subsection{Curvature convergence} Assume now that $\partial \Omega\in W^{2,q}$ for some $q\in [1,\infty)$. Then the local charts $\phi^i\in W^{2,q}(B'_{R_\Omega})$.

We differentiate twice the identity $F_m^i(y',\psi^i_m(y'))=0$ with respect to $y'_k\,y'_l$ for $k,l=1,\dots n-1$, and find
\begin{equation}\label{june77}
    \begin{split}
        \frac{\partial^2 \psi^i_m}{\partial y_k'\partial y_l'}(y')=-\bigg(\frac{\partial F^i_m(y',\psi^i_m(y'))}{\partial y_n}\bigg)^{-1}\Bigg\{ & \frac{\partial^2 F^i_m (y',\psi^i_m(y'))}{\partial y_k'\partial y_l'}  +\frac{\partial^2 F_m^i (y',\psi^i_m(y'))}{\partial y_l'\partial y_n}\,\frac{\partial \psi^i_m}{\partial y_k'}(y')+
        \\
        & +\frac{\partial^2 F_m^i (y',\psi^i_m(y'))}{\partial y_k'\partial y_n}\,\frac{\partial \psi^i_m}{\partial y_l'}(y') + 
        \\
        &
        +\frac{\partial^2 F_m^i (y',\psi^i_m(y'))}{\partial y_n\partial y_n}\,\frac{\partial \psi^i_m}{\partial y_k'}(y')\,\frac{\partial \psi^i_m}{\partial y_l'}(y')\Bigg\}\,.
    \end{split}
\end{equation}

Elementary computations and \eqref{non:Ii} show that, for $l,r=1,\dots n$, we have
\begin{equation}\label{june888}
\begin{split}
    \frac{\partial^2 F^i_m}{\partial y_r\partial y_l}(y',\psi^i_m(y'))=\sum_{j\in\mathcal{I}_i} & \Bigg\{\frac{\partial^2 f^j_m }{\partial y_r\partial y_l}(x^m)\,\xi_j(x^m)+\frac{\partial f^j_m}{\partial y_r}(x^m)\,\frac{\partial \xi_j}{\partial y_l}(x^m)
    \\ &+\frac{\partial f^j_m}{\partial y_l}(x^m)\,\frac{\partial \xi_j}{\partial y_r}(x^m)+f^j_m(x^m)\,\frac{\partial^2\xi_j}{\partial y_r\partial y_r}(x^m)\Bigg\},
    \end{split}
\end{equation}
where $x^m=(T^i)^{-1}(y',\psi^i_m(y'))$. We also have
\begin{equation}\label{june999}
    \frac{\partial^2 f^j_m}{\partial y_r\partial y_l}(x^m)=-\sum_{s,t=1}^{n-1}\frac{\partial^2 \phi^j_m}{\partial z'_s\partial z'_t}(\C^{i,j}_m y')\,\big(\mathcal{R}^i(\mathcal{R}^j)^t\big)_{rs}\big(\mathcal{R}^i(\mathcal{R}^j)^t\big)_{lt}
\end{equation}
for all $j\in \mathcal{I}_i$ such that  $x^m\in\mathrm{supp}\,\xi_j$.

Thanks to \eqref{efjmef},    \eqref{uniffff} and \eqref{edeome}, we readily find that $f^j_m(x^m)\,|\nabla \xi_j(x^m)|\to 0$ and \mbox{$f^j_m(x^m)\,|\nabla^2\xi_j(x^m)|\to 0$.}  From this, and by using \eqref{deriv:xi}, \eqref{coervic:Fm}, \eqref{nfjxm}, \eqref{lipschitz:m} and \eqref{june77}-\eqref{june999}, we obtain
\begin{equation}\label{stima:hessiana}
    |\nabla^2 \psi^i_m(y')|\leq c(n)(1+L_\Omega^5)\,\sum_{j\in\mathcal{I}_i}\bigg\{ |\nabla^2 \phi^j_m|(\C^{i,j}_m y')\,\xi_j\big((T^i)^{-1}(y',\psi^i_m(y')\big)+\frac{(1+L_\Omega)}{R_\Omega}\bigg\}\,,
\end{equation}
for all $y'\in B'_{R_\Omega-2\varepsilon_0}$, provided $m>m_0$ is large enough.

Then again, thanks to \eqref{uoaaa2} and\eqref{lip:transmap}, we may apply Proposition \ref{prop:convoluz} and infer
\begin{equation}\label{uoaaa3}
    \nabla^2 \phi^j_m(\C^{i,j}_m y')\to \nabla^2 \phi^j(\C^{i,j}y')\quad\text{for $\H^{n-1}$-a.e. $y'\in V^{i,j}$ and in $L^q(V^{i,j})$.}
\end{equation}

Finally, recalling \eqref{nulla:partiz} and \eqref{tuttovij}, the properties \eqref{coervic:Fm}, \eqref{uniffff}, \eqref{fm:der}, \eqref{puntual},  \eqref{june77}-\eqref{uoaaa3} and dominated convergence Theorem \ref{thm:dc} entail
\begin{equation*}
    \nabla^2 \psi_m^i\to M\,,\quad\text{$\H^{n-1}$-a.e. on $B'_{R_\Omega-2\varepsilon_0}$ and in $L^q(B'_{R_\Omega-2\varepsilon_0})$}\,,
\end{equation*}
for some matrix valued function $M$, which can be explictly written in terms of $\phi^j,\nabla \phi^j,\nabla^2 \phi^j$ and $\xi_j$. On the other hand,  \eqref{uniffff} and the uniqueness of the distributional limit imply that $M=\nabla^2 \phi^i$, hence \eqref{lchart:conv2} is proven.

\subsection{Proof of the isocapacitary estimate \texorpdfstring{\eqref{iscap:omm'}}{}}

In the following subsection, we will denote by $\widetilde{M}_m(h)$ the convolution of a function $h\in L^1_{loc}(\R^n)$
 with respect to the first $(n-1)$-variables, i.e.
\begin{equation*}
    \widetilde{M}_m(h)(z',z_n)=\int_{\R^{n-1}}h(x',z_n)\,\rho_m(z'-x')\,dx'\,.
\end{equation*}

We then have the following elementary lemma, which will be useful later.
\begin{lemma}\label{lem:radice}
    Let $v\in C^{\infty}_c(\R^n)$. Then, if we set \begin{equation*}
        \widetilde{v}_m\coloneqq \sqrt{\widetilde{M}_m(v^2)}\,,
\end{equation*}
we have that $\widetilde{v}_m$ is Lipschitz continuous on $\R^n$, and
\begin{equation}
    |\nabla \widetilde{v}_m|\leq c(n)\,\sqrt{\widetilde{M}_m\big(|\nabla v|^2\big)}\quad\text{a.e. on $\R^n$.}
\end{equation}
\end{lemma}
\begin{proof}
    By H\"older's inequality, for $k=1,\dots,n$ we have
    \begin{equation*}
        \bigg|\frac{\partial \widetilde{M}_m(v^2)}{\partial x_k}\bigg|=\bigg|\widetilde{M}_m\bigg(\frac{\partial v^2}{\partial x_k}\bigg)\bigg|=2\,\bigg|\widetilde{M}_m\bigg(v\,\frac{\partial v}{\partial x_k}\bigg)\bigg|\leq 2\,\sqrt{\widetilde{M}_m(v^2)}\,\sqrt{\widetilde{M}_m\bigg(\bigg|\frac{\partial v}{\partial x_k}\bigg|^2\bigg)}\,.
    \end{equation*}
Therefore, on setting $\widetilde{v}_{\varepsilon,m}\coloneqq \sqrt{\varepsilon^2+\widetilde{M}_m(v^2)}$, for all $\varepsilon\in (0,1)$ we have that
\begin{equation}\label{zach}
    |\nabla \widetilde{v}_{\varepsilon,m}|=\frac{\big|\nabla \widetilde{M}_m(v^2)\big|}{2\sqrt{\varepsilon^2+\widetilde{M}_m(v^2)}}\leq c(n)\,\frac{\sqrt{\widetilde{M}_m(v^2)}\,\sqrt{\widetilde{M}_m(|\nabla v|^2)}}{\sqrt{\varepsilon^2+\widetilde{M}_m(v^2)}}\leq c(n)\,\sqrt{\widetilde{M}_m\big(|\nabla v|^2\big)}\,.
\end{equation}
Thus, the sequence $\{\widetilde {v}_{\varepsilon,m}\}_{\varepsilon\in (0,1)}$ is uniformly bounded in $C^{0,1}_c(\R^n)$, and since $\widetilde {v}_{\varepsilon,m}\xrightarrow{\varepsilon\to 0^+} \widetilde {v}_m$ on $\R^n$, we deduce that $\widetilde {v}\in C^{0,1}_c(\R^n)$ by weak-$\ast$ compactness, and the thesis follows by letting $\varepsilon\to 0$ in \eqref{zach} and by Rademacher's Theorem. 
\end{proof}

Now let $x^0_m\in \partial \Omega_m$; then owing to \eqref{conenuto} and \eqref{iop}, there exists $i\in\{1,\dots,N\}$ such that $x^0_m\in B_{R_\Omega/8}(x^i)$. Therefore, we may write $x^0_m=(T^i)^{-1}\Big((y^0)',\psi^i_m\big((y^0)'\big)\Big)$ for some $(y^0)'\in B'_{R_\Omega/8}$, and we also set $x^0\coloneqq (T^i)^{-1}\Big((y^0)',\phi^i\big((y^0)'\big)\Big)\in \partial \Omega$. Let
$$
r_0\coloneqq \frac{R_\Omega}{C(n)\,\big(1+L_\Omega^2 \big)}\,,
$$
for some fixed constant $C(n)>1$ large enough, and consider
$r\leq r_0$, and $v\in C^\infty_c\big(B_r(x^0_m)\big)$. Then, since $B_r(x^0_m)\Subset B_{R_\Omega/4}(x^i)\Subset K^i_{2\varepsilon_0}$, we have
\begin{equation*}
    \int_{\partial \Omega_m}v^2\,|\B_{\Omega_m}|\,d\H^{n-1}=\int_{B'_{R_\Omega/4}}v^2\Big((T^i)^{-1}\big(y',\psi^i_m(y')\big)\Big)|\B_{\Omega_m}(y')|\,\sqrt{1+|\nabla\psi^i_m(y')|^2}\,dy'.
\end{equation*}

Consider the new set of indices
\begin{equation*}
    \mathbb{J}^{x^0_m}_r\coloneqq  \big\{j\in \mathcal{I}_i\,:\,B_r(x^0_m)\cap \mathrm{supp}\,\xi_j\neq \emptyset   \big\}\,.
\end{equation*}
Owing to \eqref{BBB}, \eqref{ij:transition}, \eqref{nulla:partiz},  \eqref{lipschitz:m} and the Hessian estimate \eqref{stima:hessiana}, we obtain
\begin{equation}\label{kjh1}
    \begin{split}
    &\int_{\partial \Omega_m}v^2\,|\B_{\Omega_m}|\,d\H^{n-1}  \leq\sqrt{1+L_\Omega^2}\, \int_{B'_{R_\Omega/4}}v^2\Big((T^i)^{-1}\big(y',\psi^i_m(y')\big)\Big)\,|\nabla^2 \psi^i_m(y')|\,dy'
       \\
       & \leq c(n)\,(1+L_\Omega^6)\,\sum_{j\in\mathbb{J}_r^{x^0_m}} \int_{V^{i,j}}\Bigg\{v^2\Big((T^j)^{-1}\big(\C^{i,j}_m y',\psi^j_m(\C^{i,j}_m y')\big)\Big)\times
       \\
       &\hspace{4cm} \times \xi_j\Big((T^j)^{-1}\big(\C^{i,j}_m y',\psi^j_m(\C^{i,j}_m y')\big)\Big)\,M_m\big(|\nabla^2 \phi^j|\big)(\C^{i,j}_m y')\Bigg\}\,dy'
       \\
       &\hspace{1cm}+c(n)\,\frac{(1+L^7_\Omega)}{R_\Omega}\,|\mathbb{J}_r^{x^0_m}|\int_{B'_{R/4}}v^2\Big((T^i)^{-1}\big( y',\psi^i_m( y')\big)\Big)dy'\,.
    \end{split}
\end{equation}
By using $|\mathbb{J}_r^{x^0_m}|\leq N$,  \eqref{card:NN}, \eqref{lip:car} and the results of \cite[Corollary 6.6]{accfm}, we get
\begin{equation}\label{limmitato:v}
    \begin{split}
     & \frac{(1+L^7_\Omega)}{R_\Omega}\,|\mathbb{J}_r^{x^0_m}|\int_{B'_{R_\Omega/4}}v^2\Big((T^i)^{-1}\big( y',\psi^j_m(y')\big)\Big)dy'\leq c(n)\frac{(1+L^7_\Omega)\,d_\Omega^n}{R_\Omega^{n+1}}\int_{\partial \Omega_m} v^2\,d\H^{n-1}
    \\
    &\hspace{3cm}  
 \leq \begin{dcases}
        c'(n)\,\frac{(1+L_\Omega^{25})\,d_\Omega^n}{R_\Omega^{n+1}}\,\bigg(\int_{\R^n}|\nabla v|^2\,dx\bigg)\,r\quad & \text{if $n\geq 3$}
        \\
        c\,\frac{(1+L_\Omega^{31})\,d_\Omega^n}{R_\Omega^{n+1}}\bigg(\int_{\R^2}|\nabla v|^2\,dx\bigg)\,r\,\log\Big(1+\frac{1}{r}\Big)\quad& \text{if $n=2$.}
    \end{dcases}
    \end{split}
\end{equation}
On the other hand, via the change of variables $z'=\C^{i,j}_m y'$, by making use of \eqref{lip:transmap}, \eqref{mam:mia}, and observing that $B_r(x^0_m)\Subset K^i_{2\varepsilon_0}\cap K^j_{2\varepsilon_0}$ for all $j\in \mathbb{J}^{x^0_m}_r$,  $x^0_m\in \partial \Omega_m$ and $r\leq r_0$, we find
\begin{equation}\label{kjh2}
    \begin{split}
        & \int_{V^{i,j}}\Bigg\{v^2\Big((T^j)^{-1}\big(\C^{i,j}_m y',\psi^j_m(\C^{i,j}_m y')\big)\Big)\,\xi_j\Big((T^j)^{-1}\big(\C^{i,j}_m y',\psi^j_m(\C^{i,j}_m y')\big)\Big)\,M_m\big(|\nabla^2 \phi^j|\big)(\C^{i,j}_m y')\Bigg\}\,dy'
        \\
        & \leq c(n)(1+L_\Omega^{(n-1)})\,\int_{W^{i,j}} w^2_{j,m}(z',0)\,M_m\big(|\nabla^2\phi^j|\big)(z')\,dz'\,,
    \end{split}
\end{equation}
for some open set $W^{i,j}\Subset \C^{i,j}(U^{i,j}) $, where we also set
\begin{equation*}
    w_{j,m}(z',z_n)\coloneqq v\Big((T^j)^{-1}\big(z' ,z_n+\psi^j_m( z')\big)\Big)\,.
\end{equation*}
Since $v\in C^\infty_c(B_r(x^0_m))$ and  $x^0_m=(T^j)^{-1}\Big( \C^{i,j}_m\big((y^0)'\big),\psi^j_m\big((y^0)'\big)\Big)$ for all $j\in \mathbb{J}_r^{x^0_m}$, by using \eqref{lipschitz:m} it is readily seen that
\begin{equation*}
    w_{j,m}\in C^\infty_c\bigg(B_{c(n)(1+L_\Omega^2)\,r}\Big(\C^{i,j}_m\big((y^0)'\big),0  \Big)\bigg)\,, 
\end{equation*}
and from the chain rule we find
\begin{equation}\label{gradee:1}
    |\nabla w_{j,m}(z',z_n)|\leq c(n)(1+L_\Omega^2)\,\Big|\nabla v\Big((T^{j})^{-1}\big(z',z_n+\psi_m^j(z')\big)\Big)\Big|
\end{equation}
Next, by using  Fubini-Tonelli's Theorem we obtain
\begin{equation*}
    \begin{split}\int_{W^{i,j}} w^2_{j,m}(z',0)\,M_m\big(|\nabla^2\phi^j|\big)(z')\,dz'& =\int_{W^{i,j}}w_{j,m}^2(z',0)\int_{B'_{1/m}(z')}|\nabla^2 \phi^j(\tilde{z}')|\,\rho_{m}(z'-\tilde{z}')\,d\tilde{z}'\,dz'
    \\
    &\leq \int\limits_{W^{i,j}+B'_{1/m}}|\nabla^2\phi^j(\tilde{z}')|\Big(\int_{B'_{1/m(\tilde{z}')}}w^2_{j,m}(z',0)\,\rho_m(\tilde{z}'-z')\,dz'\Big)\,d\tilde{z}'\,.
    \end{split}
\end{equation*}
We have thus found that
\begin{equation}\label{kjh3}
    \int_{W^{i,j}} w^2_{j,m}(z',0)\,M_m\big(|\nabla^2\phi^j|\big)(z')\,dz'\leq \int_{\widetilde{W}^{i,j}}\widetilde{M}_m(w^2_{j,m})(z',0)\,|\nabla^2\phi^j(z')|\,dz'\,,
\end{equation}
for some open set $\widetilde{W}^{i,j}\Subset \C^{i,j}(U^{i,j})$, provided $m>m_0$ is large enough.

Thanks to Lemma \ref{lem:radice} and inequality \eqref{uoaaa2}, we easily infer 
\begin{equation*}
    \sqrt{\widetilde{M}_m(w^2_{j,m})}\in C^{0,1}_c\bigg(B_{c(n)(1+L_\Omega^2)(r+\frac{1}{m})}\Big( \C^{i,j}\big((y^0)'\big),0\Big)\bigg)\,,
\end{equation*}
and
\begin{equation}\label{gradee:2}
    \Big|\nabla \sqrt{\widetilde{M}_m(w^2_{j,m})}\Big|\leq c(n)\,\sqrt{\widetilde{M}_m\big(|\nabla w_{j,m}|^2 \big)}\quad\text{a.e. on $\R^n$.}
\end{equation}
Finally, set
\begin{equation*}
    \tilde{h}_{j,m}(x',x_n)\coloneqq \sqrt{\widetilde{M}_m(w^2_{j,m})}\Big(T^j\big(x',x_n-\phi^j(x')\big)\Big)
\end{equation*}
so that $\tilde{h}_{j,m}$ is Lipschitz continuous on $\R^n$. Moreover, thanks to \eqref{uniffff}, for all $j\in \mathbb{J}_{r}^{x^0_m}$, we have that
\begin{equation*}
    B_{c(n)(1+L_\Omega^3)(r+\frac 1 m)}(x^0)\Subset K^i_{2\varepsilon_0}\cap K^j_{2\varepsilon_0}
\end{equation*}
for all $m>m_0$ sufficiently large and  all $r\leq r_0$, and thus we may write $x^0=(T^j)^{-1}\Big(\C^{i,j}\big((y^0)'\big),\phi^j((y^0)')  \Big)$ due to \eqref{ij:transition}.
Recalling that $\phi^j$ is $L_\Omega$-Lipschitz continous, it follows that 
\begin{equation*}
    \tilde{h}_{j,m}\in C^{0,1}_c\Big(B_{c(n)(1+L_\Omega^3)(r+\frac 1 m)}(x^0) \Big)\,,
\end{equation*}
and from the chain rule
\begin{equation}\label{gradee:3}
    \big|\nabla \tilde{h}_{j,m}(x',x_n)\big|\leq c(n)(1+L_\Omega)\, \Big|\nabla \sqrt{\widetilde{M}_m(w^2_{j,m})}(x',x_n-\phi^j(x'))\Big|\quad\text{for a.e. $x$.}
\end{equation}
Owing to \eqref{BBB} and the definition of $\widetilde{h}_{j,m}$, we have
\begin{equation}\label{kjh4}
    \begin{split}
    &\int_{\widetilde{W}^{i,j}}\widetilde{M}_m(w^2_{j,m})(z',0)\,|\nabla^2\phi^j(z')|\,dz' = \int_{\widetilde{W}^{i,j}}\tilde{h}_{j,m}^2\big((T^j)^{-1}(z',\phi^j(z'))\big)\,|\nabla^2\phi^j(z')|\,dz'
    \\
    &\hspace{0.5cm} \leq c(n)(1+L_\Omega^3)\int_{\widetilde{W}^{i,j}}\tilde{h}_{j,m}^2\big((T^j)^{-1}(z',\phi^j(z'))\big)\,\big|\B_\Omega(z')\big|\sqrt{1+|\nabla \phi^j(z')|^2}\,dz'
    \\
    &\hspace{0.5cm}  =c(n)(1+L_\Omega^3)\,\int_{\partial \Omega} \tilde{h}_{j,m}^2\big|\B_\Omega\big|\,d\H^{n-1}
    \\
    &\hspace{0.5cm} \leq c(n)(1+L_\Omega^3)\,\bigg(\sup\,\frac{\int_{\partial \Omega} h^2\,\big|\B_\Omega\big|\,d\mathcal{H}^{n-1}}{\int_{\R^n}|\nabla h|^2\,dx}\bigg)\,\int_{\R^n}|\nabla \tilde{h}_{j,m}|^2\,dx\,,
    \end{split}
\end{equation}
where the supremum above is taken over all functions $h\in C^{0,1}_c\Big(B_{c(n)(1+L_\Omega^3)(r+\frac 1 m)}(x^0) \Big)$.

Henceforth, by coupling \eqref{card:NN} and 
estimates \eqref{kjh1}-\eqref{kjh4}, for all $v\in C^\infty_c\big(B_r(x^0_m)\big)$ we obtain
\begin{equation*}
    \begin{split}
        &\int_{\partial \Omega_m}v^2\,\big|\B_{\Omega_m}\big|\,d\H^{n-1}\leq c(n)\,(1+L_\Omega^{n+4})\,\bigg(\sup\frac{\int_{\partial \Omega} h^2\,\big|\B_\Omega\big|\,d\mathcal{H}^{n-1}}{\int_{\R^n}|\nabla h|^2\,dx}\bigg)\sum_{j\in\mathbb{J}_r^{x^0_m}}\int_{\R^n}\widetilde{M}_m\big( |\nabla w_{j,m}|^2\big)\,dx
        \\
        &\hspace{11cm}+\tilde{c}\,\int_{\R^n}|\nabla v|^2 dx
        \\
        &\leq c(n)\,(1+L_\Omega^{n+4})\,\bigg(\sup\frac{\int_{\partial \Omega} h^2\,\big|\B_\Omega\big|\,d\mathcal{H}^{n-1}}{\int_{\R^n}|\nabla h|^2\,dx}\bigg)\sum_{j\in\mathbb{J}_r^{x^0_m}}\int_{\R^n} |\nabla w_{j,m}|^2\,dx+\tilde{c}\,\int_{\R^n}|\nabla v|^2 dx
        \\
        &\leq c(n)\,(1+L_\Omega^{n+8})\,N\,\bigg(\sup\frac{\int_{\partial \Omega} h^2\,\big|\B_\Omega\big|\,d\mathcal{H}^{n-1}}{\int_{\R^n}|\nabla h|^2\,dx}\bigg)\,\int_{\R^n} |\nabla v|^2\,dx+\tilde{c}\,\int_{\R^n}|\nabla v|^2 dx
        \\
        &\leq c'(n)\,(1+L_\Omega^{n+8})\frac{d_\Omega^n}{R_\Omega^n}\,\bigg(\sup\frac{\int_{\partial \Omega} h^2\,\big|\B_\Omega\big|\,d\mathcal{H}^{n-1}}{\int_{\R^n}|\nabla h|^2\,dx}\bigg)\,\int_{\R^n}|\nabla v|^2 dx+\tilde{c}\,\int_{\R^n}|\nabla v|^2 dx\,,
    \end{split}
\end{equation*}
where in the second inequality we made use of Fubini-Tonelli's Theorem, the supremum above is taken over all $h\in C^{0,1}_c\Big(B_{c(n)(1+L_\Omega^3)(r+\frac 1 m)}(x^0) \Big)$, and we set
\begin{equation}\label{c:hatil}
\tilde{c}=\tilde{c}(n,L_\Omega,R_\Omega,d_\Omega,r)=
    \begin{dcases}
        c(n)\frac{(1+L_\Omega^{25})\,d_\Omega^n}{R_\Omega^{n+1}}\,r  \quad & \text{if $n\geq 3$ }
        \\
        c(n)\frac{(1+L_\Omega^{31})\,d_\Omega^n}{R_\Omega^{n+1}}\,r\,\log\Big(1+\frac{1}{r}\Big)\quad & \text{if $n=2$.}
    \end{dcases}
\end{equation}
Therefore, for all $x^0_m\in \partial \Omega_m$, $r\leq r_0$, we have found
\begin{equation*}
\begin{split}
    \sup_{v\in C^{\infty}_c(B_r(x^0_m))} & \frac{\int_{\partial \Omega_m} v^2\,\big|\B_{\Omega_m}\big|\, d\mathcal{H}^{n-1}}{\int_{\R^n}|\nabla v|^2\,dx} 
    \\
& \leq  \frac{c(n)\,(1+L_\Omega^{n+8})\,d_\Omega^n}{R_\Omega^{n}}\,\Bigg(\displaystyle 
\sup_{
\begin{tiny}
 \begin{array}{c}{x^0\in\partial \Omega}\\
v\in C^{0,1}_c\big(B_{c(n)(1+L_\Omega^3)(r+1/m)}(x^0)\big)
 \end{array}
  \end{tiny}
}
\frac{\int_{\partial \Omega} v^2\,\big|\B_{\Omega}\big|\,d\mathcal{H}^{n-1}}{\int_{\R^n}|\nabla v|^2\,dx}\Bigg)+\tilde{c}\,.
\end{split}
\end{equation*}
From this, \eqref{c:hatil} and the isocapacitary equivalence \cite[Theorem 2.4.1]{maz}, we finally obtain the desired estimate
\begin{equation}\label{is:capprec}
    \mathcal{K}_{\Omega_m}(r)\leq 
\begin{dcases}
     \frac{c(n)\,(1+L_\Omega^{n+8})\,d_\Omega^n}{R_\Omega^{n}}\,\mathcal{K}_{\Omega}\Big( c(n)(1+L_\Omega^3)(r+\tfrac{1}{m})\Big)+\frac{c(n)\,(1+L_\Omega^{25})\,d_\Omega^n}{R_\Omega^{n+1}}\,r  \quad & \text{if $n\geq 3$ }
     \\
      \frac{c(n)\,(1+L_\Omega^{n+8})\,d_\Omega^n}{R_\Omega^{n}}\,\mathcal{K}_{\Omega}\Big( c(n)(1+L_\Omega^3)(r+\tfrac{1}{m})\Big)+\frac{c(n)\,(1+L_\Omega^{31})\,d_\Omega^n}{R_\Omega^{n+1}}\,r\,\log\Big(1+\frac{1}{r}\Big)  \quad & \text{if $n=2$,}
\end{dcases}
\end{equation}
for all $r\leq r_0$ and $m>m_0$, and the proof is complete.

\section*{Acknowledgments}
\noindent
I would like to thank professors A. Cianchi, G. Ciraolo and A. Farina for suggesting the problem, and for useful discussions and observations on the topic. 

The author has been partially supported by the ``Gruppo Nazionale per l'Analisi Matematica, la Probabilit\`a e le loro Applicazioni'' (GNAMPA) of the ``Istituto Nazionale di Alta Matematica'' (INdAM, Italy).
\medskip

\par\noindent {\bf Data availability statement.} Data sharing not applicable to this article as no datasets were generated or analysed during the current study.

\end{document}